\documentclass[12pt]{amsart}

\usepackage{amsmath}
\usepackage{amssymb}
\usepackage{array}
\usepackage{enumitem}
\usepackage{url}
\usepackage{verbatim} 

\input cyracc.def
\newfam\cyrfam
\font\twelvecyr=wncyr10 scaled 1200 
\def\cyr{\fam\cyrfam\twelvecyr\cyracc}

\theoremstyle{plain}
\newtheorem{theorem}{Theorem}[section]
\newtheorem{lemma}[theorem]{Lemma}
\newtheorem{proposition}[theorem]{Proposition}
\newtheorem{corollary}[theorem]{Corollary}

\theoremstyle{definition}
\newtheorem*{definition}{Definition}
\theoremstyle{remark}
\newtheorem*{remark}{Remark}
\newtheorem*{example}{Example}
\newtheorem*{notation}{Notation}
\newtheorem*{acknowledgment}{Acknowledgment}

\numberwithin{equation}{section}

\newcommand{\seclabel}[1]{\label{sec:#1}}   
\newcommand{\sublabel}[1]{\label{sub:#1}}   
\newcommand{\thmlabel}[1]{\label{thm:#1}}   
\newcommand{\lemlabel}[1]{\label{lem:#1}}   
\newcommand{\prplabel}[1]{\label{prp:#1}}   
\newcommand{\eqnlabel}[1]{\label{eqn:#1}}   

\newcommand{\subref}[1]{\ref{sub:#1}}   
\newcommand{\thmref}[1]{\ref{thm:#1}}   
\newcommand{\prpref}[1]{\ref{prp:#1}}   
\newcommand{\eqnref}[1]{\eqref{eqn:#1}} 

\newcommand{\bA}{\mathbb{A}}
\newcommand{\bB}{\mathbb{B}}
\newcommand{\bK}{\mathbb{K}}
\newcommand{\bP}{\mathbb{P}}
\newcommand{\bR}{\mathbb{R}}
\newcommand{\bC}{\mathbb{C}}

\newcommand{\cX}{\mathcal{X}}
\newcommand{\cY}{\mathcal{Y}}

\newcommand{\rr}{\mathbf{r}}
\newcommand{\bl}{\mathbf{l}}

\newcommand{\Gl}{\mathop{\mathrm{Gl}}\nolimits}
\newcommand{\Hom}{\mathrm{Hom}}

\newcommand{\End}{\mathrm{End}}
\newcommand{\Iso}{\mathrm{Iso}}
\newcommand{\Gras}{\mathrm{Gras}}

\newcommand{\id}{\mathrm{id}}

\newcommand{\setof}[2]{\{ #1 \mid #2\}}
\newcommand{\Bigsetof}[2]{\begin{Bmatrix} #1 \,\Big|\, #2 \end{Bmatrix}}
\newcommand{\myand}{\enspace\text{and}\enspace}
\newcommand{\myor}{\enspace\text{or}\enspace}
\newcommand{\inv}{^{-1}}

\begin{document}

\title[Associative Geometries. I]{Associative Geometries. I:
Torsors, linear relations and Grassmannians}

\author{Wolfgang Bertram}

\address{Institut \'{E}lie Cartan Nancy \\
Nancy-Universit\'{e}, CNRS, INRIA, \\
Boulevard des Aiguillettes, B.P. 239, \\
F-54506 Vand\oe{}uvre-l\`{e}s-Nancy, France}

\email{\url{bertram@iecn.u-nancy.fr}}

\author{Michael Kinyon}

\address{Department of Mathematics \\
University of Denver \\
2360 S Gaylord St \\
Denver, Colorado 80208 USA}

\email{\url{mkinyon@math.du.edu}}



\begin{abstract}
We define and investigate a geometric object, called
an \emph{associative geometry}, corresponding to an
associative algebra (and, more generally, to an associative pair).
Associative geometries combine aspects of \emph{Lie groups} and of
\emph{generalized projective geometries}, where the former
correspond to the Lie product of an associative algebra and the latter
to its Jordan product.  A further development of the theory encompassing
\emph{involutive} associative algebras will be given in subsequent
work \cite{BeKi09}.
\end{abstract}

\maketitle

\section*{Introduction}
\seclabel{intro}

What is the geometric object corresponding to an \emph{associative} algebra?
The question may come as a bit of a surprise: the philosophy of Noncommutative Geometry
teaches us that, as soon as an algebra becomes noncommutative, we should stop looking
for associated point-spaces, such as manifolds or varieties.
Nevertheless, we raise this question, but aim at something different
than Noncommutative Geometry: we do not try to generalize
the relation between, say, commutative associative algebras and algebraic varieties,
but rather look for an analog of the one between \emph{Lie algebras} and \emph{Lie groups}.
Namely, every associative algebra $\bA$ gives rise to a Lie algebra $\bA^-$ with commutator
bracket $[x,y]=xy-yx$, and thus can be seen as a ``Lie algebra with some additional structure''.
Since the geometric object corresponding to a Lie algebra should be a Lie
group (the unit group $\bA^\times$, in this case), the object
corresponding to the \emph{associative} algebra, called an
``associative geometry'',  should be some kind of ``Lie group with additional structure''.
To get an idea of what this additional structure might be, consider the decomposition
\[
xy = \frac{xy+yx}{2} + \frac{xy-yx}{2} =: x \bullet y + \frac{1}{2}\lbrack x,y\rbrack
\]
of the associative product into its symmetric and skew-symmetric parts.
The symmetric part is a \emph{Jordan} algebra, and the additional structure
will be related to the geometric object corresponding to the Jordan part.
As shown in \cite{Be02}, the ``geometric Jordan object'' is a \emph{generalized
projective geometry}. Therefore, we expect an associative geometry to be
some sort of mixture of projective geometry and Lie groups.
Another hint is given by the notion of \emph{homotopy} in associative algebras:
an associative product $xy$ really gives rise to a family of associative
products
\[
x \cdot_a y := xay
\]
for any fixed element $a$, called the \emph{$a$-homotopes}.
Therefore we should rather expect to deal with a whole family of Lie
groups, instead of looking just at \emph{one} group
corresponding to the choice $a =1 $.

\subsection{Grassmannian torsors}
\sublabel{group_grass}

The following example gives a good idea of the kind of geometries
we have in mind.
Let $W$ be a vector space or module over a commutative field or ring
$\bK$, and for a subspace $E \subset W$, let $C^E$ denote the set of
all subspaces of $W$ complementary to $E$.
It is known that $C^E$ is, in a natural way, an affine space over $\bK$.
%
%
We prove that a similar statement is true
for arbitrary intersections $C^E \cap C^F$ (Theorem 1.2):
they are either empty, or they carry a natural ``affine'' group
structure. By this we mean that, after fixing an arbitrary
element $Y \in C^E \cap C^F$, there is a natural (in general
noncommutative) group structure on $C^E \cap C^F$ with
unit element $Y$.
The construction of the group law is very simple: for
$X,Z \in C^E \cap C^F$, we let
$X \cdot Z := (P^E_X - P^Z_F)(Y)$, where, for any complementary pair
$(U,V)$, $P^U_V$ is the projector onto $V$ with kernel $U$.
Since $X \cdot Z$ indeed depends on $X,E,Y,F,Z$, we write it
also in quintary form
\begin{equation}
\eqnlabel{first_gamma}
\Gamma(X,E,Y,F,Z):= (P^E_X - P^Z_F)(Y).
\end{equation}
The reader is invited to prove the group axioms by direct calculations.
The proofs are elementary, however, the associativity
of the product, for example, is not obvious at a first glance.

Some special cases, however, are relatively clear. If $E=F$, and if
we then identify a subspace $U$ with the projection $P^E_U$, then it
is straightforward to show that the expression $\Gamma(X,E,Y,E,Z)$
in $C^E$ is equivalent to the expression $P^E_X - P^E_Y + P^E_Z$ in
the space of projectors with kernel $E$, and we recover the
classical affine space structure on $C^E$ (see Theorem \thmref{geom_props}).
On the other hand, if
$E$ and $F$ happen to be mutually complementary, then any common
complement of $E$ and $F$ may be identified with the graph of
a bijective linear map $E \to F$, and hence $C^E \cap C^F$ is identified
with the set $\Iso(E,F)$ of linear isomorphisms between $E$ and $F$.
Fixing an origin $Y$ in this set fixes an identification of $E$ and $F$,
and thus identifies $C^E \cap C^F$ with the general linear group $\Gl_{\bK}(E)$.

Summing up, the collection of groups $C^E \cap C^F$, where  $(E,F)$
runs through $\Gras(W)\times \Gras(W)$, the direct product of the Grassmannian
of $W$ with itself, can be seen as some kind of interpolation, or deformation
between general linear groups and vector groups, encoded in $\Gamma$.
The quintary map $\Gamma$ has remarkable properties that will
lead us to the axiomatic definition of associative geometries.

\subsection{Torsors and semitorsors}
\sublabel{torsors}

To eliminate the dependence of the group structures $C^E\cap C^F$ on the
choice of unit element $Y$, we now recall the
``affine'' or ``base point free'' version of the concept of group.
There are several equivalent versions,
going under different names such as
\emph{heap}, \emph{groud}, \emph{flock}, \emph{herd},
\emph{principal homogeneous space}, \emph{abstract coset},
\emph{pregroup} or others. We use what seems to be the most currently
fashionable term, namely \emph{torsor}.
The idea is quite simple (see Appendix A for details):
if, for a given group $G$ with unit element $e$, we want to ``forget
the unit element'', we consider $G$ with the ternary product
\[
G \times G \times G \to G; \quad (x,y,z) \mapsto (xyz):=xy\inv z.
\]
As is easily checked,
this map has the following properties: for all $x,y,z,u,v \in G$,
\begin{align*}
(xy(zuv))=((xyz)uv)\,, \tag{G1} \\
(xxy)=y=(yxx)\,. \tag{G2}
\end{align*}
Conversely, given a set $G$ with a ternary composition having these
properties, for any element $x \in G$ we get a group law on $G$
with unit $x$ by letting $a \cdot_x b:= (axb)$ (the inverse of
$a$ is then $(xax)$) and such that $(abc)=ab\inv c$ in this group.
(This observation is stated explicitly by Certaine in
\cite{Cer43}, based on earlier work Pr\"ufer, Baer, and others.)
Thus the affine concept of the group $G$ is a set $G$ with a
ternary map satisfying (G1) and (G2); this is precisely the
structure we call a \emph{torsor}.

One advantage of the torsor concept, compared to other, equivalent
notions mentioned above, is that it admits two natural and important extensions.
On the one hand, a direct check shows that in any torsor the relation
\[
(xy(zuv))= (x(uzy)v)=((xyz)uv)\,, \tag{G3}
\]
called the \emph{para-associative law}, holds (note the reversal of
arguments in the middle term). Just as groups are generalized by semigroups,
torsors are generalized by \emph{semitorsors} which are simply sets
with a ternary map satisfying (G3). It is already known that this
concept has important applications in geometry and algebra. The idea
can be traced back at least as far work 
of V.V.\ Vagner, \emph{e.g.} \cite{Va66},.

On the other hand, \emph{restriction to the diagonal} in a torsor gives
rise to an interesting product $m(x,y):=(xyx)$.
The map $\sigma_x: y \mapsto m(x,y)$ is just inversion in the group $(G,x)$.
If $G$ is a Lie torsor (defined in the obvious way), then $(G,m)$ is a
\emph{symmetric space} in the sense of Loos \cite{Lo69}.

\subsection{Grassmannian semitorsors}
\sublabel{semitorsor_grass}

One of the remarkable properties of the quintary map $\Gamma$ defined
above is that it admits an ``algebraic continuation'' from the
subset $D(\Gamma) \subset \cX^5$ of $5$-tuples from the Grassmannian
$\cX = \Gras(W)$ where it was initially defined to \emph{all} of $\cX^5$.
The definition given above requires that the pairs $(E,X)$ and
$(F,Z)$ are complementary. On the other hand, fixing an arbitrary
complementary pair $(E,F)$, there is another natural ternary product:
with respect to the decomposition $W=E \oplus F$, subspaces $X,Y,Z,\ldots$
of $W$ can be considered as \emph{linear relations between $E$ and $F$}, and
can be composed as such: $ZY\inv X$ is again a linear relation between $E$
and $F$.
Since $ZY\inv X$ depends on $E$ and on $F$, we get another map
\[
\Gamma(X,E,Y,F,Z):= X Y\inv Z .
\]
Looking more closely at the definition of this map, one realizes that
there is a natural extension of its domain for \emph{all} pairs $(E,F)$,
and that on $D(\Gamma)$ this new definition of $\Gamma$ coincides with
the earlier one given by \eqnref{first_gamma} (Theorem \ref{MainTheorem}).
Moreover, for any fixed pair $(E,F)$, the ternary product
\[
(XYZ):= \Gamma(X,E,Y,F,Z)
\]
turns the Grassmannian $\cX$ into a semitorsor.
The list of remarkable properties of $\Gamma$ does not end here --
we also have \emph{symmetry properties} with respect to the
Klein $4$-group acting on the variables $(X,E,F,Z)$,
certain interesting \emph{diagonal values} relating the map $\Gamma$
to \emph{lattice theoretic properties} of the Grassmannian (Theorem \ref{lattice})
 as well as
\emph{self-distributivity} of the product, reflecting the fact
that all partial maps of $\Gamma$ are \emph{structural}, i.e.,
compatible with the whole structure (Theorem \ref{structure}).
Together, these properties can be used to give an axiomatic
definition of an associative geometry (Chapter 3).

\subsection{Correspondence with associative algebras and pairs}
\sublabel{correspondence}

Taking the Lie functor for Lie groups as model, we wish to define
a multilinear \emph{tangent} \emph{object} attached to an associative geometry
at a given base point. A \emph{base point} in $\cX$ is a fixed
complementary (we say also \emph{transversal}) pair $(o^+,o^-)$.
The pair of abelian groups  $(\bA^+,\bA^-):=(C^{o^-},
C^{o^+})$ then plays the r\^ole of a pair of ``tangent spaces'', and
the r\^ole of the Lie bracket is taken by the following pair of  maps:
\[
f^\pm: \bA^\pm \times \bA^\mp \times \bA^\pm \to \bA^\pm ; \quad
(x,y,z) \mapsto \Gamma(x,o^+,y,o^-,z) .
\]
One proves that $f^\pm$ are trilinear (Theorem 3.4). Since the maps
$f^\pm$ come from a semitorsor, they
form an \emph{associative pair}, i.e., they satisfy the para-associative
law (see Appendix B).
Conversely, one can construct, for every associative pair, an
associative geometry having the given pair as tangent object (Theorem 3.5).
The prototype of an associative pair are operator spaces,
$(\bA^+,\bA^-)=(\Hom(E,F),\Hom(F,E))$, with trilinear products
$f^+(X,Y,Z)=XYZ$, $f^-(X,Y,Z)=ZYX$. They correspond precisely to
Grassmannian geometries $\cX = \Gras(E \oplus F)$ with base point
$(o^+,o^-)=(E,F)$.

\emph{Associative unital algebras} are  associative pairs of
the form $(\bA,\bA)$; in the
example just mentioned, this corresponds to the special case $E=F$.
In this example, the unit element $e$ of $\bA$ corresponds to the diagonal
$\Delta \subset E \oplus E$, and the subspaces $(E,\Delta,F)$
are mutually complementary. On the geometric level, this translates to the
existence of a \emph{transversal triple} $(o^+,e,o^-)$.
Thus the correspondence between associative geometries and associative
pairs contains as a special case the one between associative geometries
with transversal triples and unital associative algebras (Theorem 3.7).

\subsection{Further topics}
\sublabel{further}

Since associative algebras play an important r\^ole in modern mathematics,
the present work is related to a great variety of topics and leads
to many new problems located at the interface of geometry and algebra.
We mention some of them in the final chapter of this work, without
attempting to be exhaustive. In particular, in part II of
this work (\cite{BeKi09}) we will extend the theory to
\emph{involutive} associative algebras (topic (2) mentioned in Chapter 4).

\begin{acknowledgment}
We would like to thank Boris Schein for enlightening us as to the history
of the torsor/groud concept. The second named author would like to thank the
Institut \'Elie Cartan Nancy for hospitality when part of this work was
carried out.
\end{acknowledgment}

\begin{notation}
Throughout this work, $\bK$ denotes  a commutative unital ring and
$\bB$ an associative unital $\bK$-algebra, and we will consider
\emph{right} $\bB$-modules $V,W,\ldots$.
We think of $\bB$ as ``base ring'', and the letter $\bA$ will be
reserved for other associative $\bK$-algebras such as $\End_\bB(W)$.
For a first reading, one may assume that $\bB = \bK$;
only in Theorem 3.7 the possibility to work over non-commutative base
rings becomes crucial.

When viewing submodules as elements of a Grassmannian, we will frequently
use lower case letters to denote them, since this matches our later notation
for abstract associative geometries. However, we will also sometimes switch
back to the upper case notation we have already used whenever it adds
clarity.
\end{notation}

\section{Grassmannian torsors}

The \emph{Grassmannian} of a right $\bB$-module $W$ is the set
$\cX = \Gras(W)=\Gras_\bB(W)$ of all $\bB$-submodules of $W$.
If $x \in \cX$ and $a \in \cX$ are complementary ($W =x \oplus a$),
we will write $x \top a$ and call the pair $(x,a)$ \emph{transversal}.
We write $C_a := a^\top := \{ x \in \cX | \, x \top a \}$ for the set
of all complements of $a$ and
\[
C_{ab} := a^\top \cap b^\top
\]
for the set of common complements of $a,b \in \cX$.
We think of $a^\top$ and $C_{ab}$ (which may or may not be empty)
as ``open chart domains'' in $\cX$. The following discussion makes this
more precise.

\subsection{Connected components and base points}
\sublabel{connected}

\subsubsection*{Connectedness}
We define an equivalence relation in $\cX$:
$x \sim y$ if there is a finite sequence of ``charts joining $x$ and $y$'',
i.e.:
$\exists a_0,a_1,\ldots,a_k$ such that $a_0=x$, $a_k=y$ and
\[
 \forall i=0,\ldots, k-1: \quad  C_{a_i,a_{i+1}} \not= \emptyset.
\]
The equivalence classes of this relation are called
\emph{connected components of $\cX$}.
We say that $x \in \cX$ is \emph{isolated} if its connected component
is a singleton.
If $\bB = \bK$ and $\bK$ is a field, then connected components
are never reduced to a point (unless $x=0$ or $x=W$).
For instance, the connected components of $\Gras(\bK^n)$
are the Grassmannians $\Gras_p(\bK^n)$ of subspaces of a fixed dimension
$p$ (indeed, two subspaces of the same dimension $p$ in $\bK^n$ always
admit a common complement, hence sequences of length $1$ always suffice in the
above condition.

\subsubsection*{Base points and pair geometries}
A \emph{base pair} or \emph{base point}
 in $\cX$ is a fixed transversal pair, often
denoted by $(o^+,o^-)$.
If $(o^+,o^-)$ is a base point, then in general $o^+$ and
$o^-$ belong to different connected components, which we
denote by $\cX^+$ and $\cX^-$.
For instance, in the Grassmann geometry $\Gras(\bK^n)$ over a field
$\bK$, if $o^+$ is of dimension $p$, then $o^-$ has to be of dimension
$q=n-p$, and hence they belong to different components unless $p=q=
\frac{n}{2}$.

More generally, we may consider certain subgeometries of $\cX$, namely
pairs $(\cX^+,\cX^-)$ of sets $\cX^\pm \subset \cX$ 
such that, for every $x \in \cX^\pm$, the set $x^\top$ is a nonempty
subset of $\cX^\mp$.
We refer to $(\cX^+,\cX^-)$ as a \emph{pair geometry}.

For instance, if $W=\bB$,
then $\cX$ is the \emph{space of right ideals in $\bB$}. Fix
 an idempotent  $e \in \bB$ and let
$o^+:= e \bB$, $o^- = (1-e) \bB$ and
$\cX^\pm$ the set of all right ideals in $\bB$ that are isomorphic
to $o^\pm$ and have a complement isomorphic to $o^\mp$.
Then $(\cX^+,\cX^-)$ is a pair geometry.

\subsubsection*{Transversal triples and spaces of the first kind}
We say that $\cX$ is \emph{of the first kind} if there exists
a triple $(a,b,c)$ of mutually transversal elements, and
\emph{of the second kind} else. Clearly,
 $a,b,c$ then all belong to the same connected component
of $\cX$; taking $(a,c)$ as base point $(o^+,o^-)$, we thus
have $\cX^+ = \cX^-$.
Note that $W= a \oplus c$ with $a \cong b \cong c$, so $W$
is ``of even dimension''.
For instance, the Grassmann geometry $\Gras(\bK^n)$ over a field
$\bK$ is of the first kind if and only if $n$ is even,
and the preceding example of a pair geometry of right ideals is
of the first kind if and only if $o^+$ and $o^-$ are isomorphic
as $\bB$-modules. In other words, $\bB$ is a direct sum
of two copies of some other algebra, and $\cX^+=\cX^-$ is
the \emph{projective line over this algebra}, \textit{cf.} \cite{BeNe05}.

\subsection{Basic operators and the product map $\Gamma$}
\sublabel{basic}

If $x$ and $a$ are two complementary $\bB$-submodules,
let $P_x^a:W \to W$ be the
projector onto $x$ with kernel $a$. Since $a$ and $x$ are
$\bB$-modules, this map is $\bB$-linear.
The relations
\[
P_x^a \circ P_z^a = P_x^a, \quad \quad P_x^a \circ P_x^b = P_x^b,
\quad \quad P^z_b \circ P^a_z = 0
\]
will be constantly used in the sequel.
For a $\bB$-linear map $f:W \to W$, we denote by $[f]:=f\mod \bK^\times$
be its projective class with respect to invertible scalars from $\bK$.
By $1$ we denote the (class of) the identity operator on $W$.
We define the following operators :
if $a \top x$ and $z \top b$, define the \emph{middle
multiplication operator} (motivation for this terminology will be given below)
\[
M_{xabz}:= [P_x^a - P^z_b],
\]
and if $a \top x$ and $y \top b$, define the
\emph{left multiplication operator}
\[
L_{xayb}:=[1 - P_a^x P_y^b]
\]
and if $a \top y$ and $z \top b$, define the \emph{right
multiplication operator}
\[
R_{aybz}:=L_{zbya}=[1-P^z_b P^a_y].
\]
For a scalar $s \in \bK$ and a transversal pair $(x,a)$, the
\emph{dilation operator} is defined by
\[
\delta^{(s)}_{xa}:= [s P_a^x + P_x^a] = [1 - (1-s)P_a^x] =
[s 1 + (1-s) P_x^a].
\]
Note that the dilation operator for the scalar $-1$ is also
a middle multiplication operator:
\[
\delta^{(-1)}_{xa} = [- P_a^x + P_x^a] = M_{xaax},
\]
and it is induced by a reflection with respect to a subspace.
Also, $\delta^{(1)}_{xa}=1$ and $\delta^{(0)}_{xa}=[P^a_x]$.

\begin{proposition}
\prplabel{1.1}
{\ } 
\begin{enumerate}[label=\roman*\emph{)},leftmargin=*]
\item (Symmetry)
 $M_{xabz}$ is invariant under permutations of indices
by the Klein $4$-group:
\[
M_{xabz} = M_{axzb} = M_{bzxa} = M_{zbax} .
\]
\item (Fundamental Relation)
Whenever $u,x \top a$ and $v,z \top b$,
\[
R_{aubz} L_{xavb}  = M_{xabz}M_{uabv} =L_{xavb}R_{aubz}  .
\]
\item (Diagonal values)
If $x \in C_{ab}$,
\[
L_{xaxb}=1=R_{axbx},
\]
and, for all $u \in C_{ab}$ and $z \top b$,
\[
M_{uabz}(u) = z = R_{aubz}(u).
\]
\item (Compatibility)
If $x \top a$, $y \top b$, $z  \top b$ and $C_{ab}$ is not empty,
then
\[
L_{xayb}(z)=M_{xabz}(y),
\]
and if $x \top a$, $z \top b$, $y  \top a$ and $C_{ab}$ is not empty,  then
\[
M_{xabz}(y)=R_{aybz}(x).
\]
\item (Invertibility)
Let $(x,a,y,b,z) \in \cX^5$ such that $x,y,z \in C_{ab}$.
Then the operators
\[
L_{xayb}, \quad M_{xabz}, \quad R_{aybz}
\]
are invertible, with inverse operators, respectively,
\[
L_{yaxb}, \quad M_{zabx}, \quad R_{azby}.
\]
\end{enumerate}
\end{proposition}

\begin{proof}
(i):
\[
M_{xabz}= [P_x^a - P^z_b] =[P^z_b-P_x^a]=M_{bzxa}.
\]
Since this is the only place where we really use that $[f]=[-f]$,
for simplicity of notation, we henceforth omit the brackets $[ \quad ]$.
\[
M_{zbax} = P_z^b - P^x_a=(1-P_b^z)-(1-P_x^a)=M_{xabz}
\]

(ii):
Using in the second line that $P_a^x P_v^b P_b^z P_u^a =0$:
\begin{eqnarray*}
L_{xavb}R_{aubz} & =& (1-P_a^x P_v^b)(1- P_b^z P_u^a) \\
&= &1 - P_a^x P_v^b- P_b^z P_u^a  \cr
&= &1 - (1-P_x^a)(1-P_b^v) - P_b^z P_u^a \cr
&= &P_x^a + P_b^v - P_x^a P_b^v - P_b^z P_u^a \cr
&= &(P_x^a - P^z_b)(P_u^a - P_b^v) \cr
&= &M_{xabz}M_{uabv} .
\end{eqnarray*}
The relation  $R_{aubz} L_{xavb}  = M_{xabz}M_{uabv}$ now follows from
(i).

(iii): $L_{xaxb}=1=R_{axbx}$ is clear.
Fix an element $u \in C_{ab}$. Then, for all $z \top b$,
\[
M_{uabz}(u) = M_{zbau}(u)=(P^b_z - P^u_a)(u)=P^b_z(u)=z
\]
since both $u$ and $z$ are complements of $b$. Similarly,
\[
R_{aubz}(u)= (1-P^z_b P^a_u)(u)=(1-P^z_b)(u)=P^b_z(u)=z.
\]

(iv):
By (ii),
$M_{xabz}M_{uaby} =L_{xayb}R_{aubz}$. Apply this operator to
$u \in C_{ab}$ and use that, by (iii), $M_{uaby}(u)=y$ and $R_{aubz}(u)=z$.
One gets
\[
M_{xabz}(y)=
M_{xabz}M_{uaby}(u) = L_{xayb}R_{aubz}(u) = L_{xayb}(z).
\]
Via the symmetry relation (i), the second equality can also be written
$M_{zbax}(y)=L_{zbya}(x)$ and hence is equivalent to the first one.

(v):
Since $L_{xaxb} = 1 = R_{axbx}$, the fundamental relation (ii) implies
$M_{xabz} M_{zaby}=L_{xayb}$ and
\[
M_{xabz} M_{zabx}=L_{xaxb}=1,
\]
hence $M_{xabz}$ is invertible with inverse $M_{zabx}$.
The other relations are proved similarly.
\end{proof}

\begin{remark}
We will prove in Chapter 2 by different methods that the
assumption $C_{ab} \not= \emptyset$ in (iv) is unnecessary.
\end{remark}

\begin{definition}[\textbf{of the product map} $\Gamma$]
We define a map $\Gamma:D(\Gamma) \to \cX$ on the following
domain of definition: let
\begin{align*}
D_L &:= \setof{(x,a,y,b,z) \in \cX^5}
{x \top a\myand y \top b} \\
D_R &:= \setof{(x,a,y,b,z) \in \cX^5}
{y \top a\myand z \top b} \\
D_M &:= \setof{(x,a,y,b,z) \in \cX^5}
{x \top a,\, z \top b\myand C_{ab} \not= \emptyset} \\
D(\Gamma) &:= D_L \cup D_R \cup D_M\,,
\end{align*}
and define $\Gamma:D(\Gamma) \to \cX$ by
\[
\Gamma(x,a,y,b,z):= 
\begin{cases}
L_{xayb}(z) & \text{if}\quad (x,a,y,b,z) \in D_L \\
R_{aybz}(x) & \text{if}\quad (x,a,y,b,z) \in D_R \\
M_{axbz}(y) & \text{if}\quad (x,a,y,b,z) \in D_M\,.
\end{cases}
\]
\end{definition}

\noindent This is well-defined: if
$(x,a,y,b,z) \in D_L \cap D_R$, then
$y \in C_{ab}$, hence $C_{ab}$ is not empty
and the preceding proposition implies that
\[
L_{xayb}(z)=M_{xabz}(y)=R_{aybz}(x).
\]
Similar remarks apply to the cases $(x,a,y,b,z) \in D_L \cap D_M$
or  $(x,a,y,b,z) \in D_R \cap D_M$.
The quintary map $\Gamma$ explains our terminology and notation:
$L_{xayb}$ is the left multiplication operator, acting on the last
argument $z$, and similarly
$R$ and $M$ denote right and middle multiplication operators.
From the definition it follows easily that the symmetry relation
\[
\Gamma(x,a,y,b,z)=\Gamma(z,b,y,a,x)
\]
holds for all $(x,a,y,b,z) \in D(\Gamma)$.
On the other hand, the relation
\[
\Gamma(x,a,y,b,z)=\Gamma(a,x,y,z,b)
\]
holds if $(x,a,y,z,b) \in D_M$; but at present it is somewhat
complicated to show that this relation is valid on all of $D(\Gamma)$
(this will follow from the results of Chapter 2).
As to the ``diagonal values'', for  $x \in C_{ab}$ we have
\[
\Gamma(x,a,x,b,z)=z=\Gamma(z,b,x,a,x)\,.
\]
If we assume just $a \top x$ and $b \top z$, then
we can only say in general that
\[
\Gamma(x,a,x,b,z) = (1 - P^z_b P^a_x)(x) = P^b_z(x) \subset z\,.
\]
If $a,b \top x$ and $b \top  z$, then, thanks to the symmetry relation
$M_{xabz} = M_{axzb}$,
\begin{equation}
\eqnlabel{diagonalv}
M_{xabz}(a)= \Gamma(x,a,a,b,z)=\Gamma(a,x,a,z,b)=b\,.
\end{equation}

\begin{definition}[\textbf{of the dilation map} $\Pi_s$]
Fix $s \in \bK$. Let
\[
D(\Pi_s) := \setof{(x,a,z) \in \cX^3}
{x \top a \myor z \top a}
\]
and define a ternary map $\Pi_s:D(\Pi_s) \to \cX$ by
\[
\Pi_s(x,a,z):=
\begin{cases}
\delta^{(s)}_{xa}(z) & \text{if}\quad x \top a \\
\delta^{(1-s)}_{za}(x) & \text{if}\quad z \top a  .
\end{cases}
\]
\end{definition}

\noindent As above, this map is well-defined.
The symmetry relation
\[
\Pi_s(x,a,y)=\Pi_{1-s}(y,a,x)
\]
follows easily from the definition.
Note that,
if $s$ is invertible in $\bK$ and $x \top a$,
then the dilation operator $\delta^{(s)}_{xa}$ is invertible
with inverse  $\delta^{(s\inv)}_{xa}$.

\subsection{Grassmannian torsors and their actions}

Recall from \S\subref{torsors} and Appendix A the definition and elementary properties of
\emph{torsors}.

\begin{theorem}
\thmlabel{geom_props}
The Grassmannian geometry $(\cX;\Gamma,\Pi_r)$ defined
in the preceding subsection has the following properties:
\begin{enumerate}[label=\roman*\emph{)},leftmargin=*]
\item For $a,b \in \cX$ fixed, $C_{ab}$ with product
\[
(xyz):=\Gamma(x,a,y,b,z)
\]
is a torsor (which will be denoted by $U_{ab}$).
In particular, for a  triple
$(a,y,b)$ with $y \in C_{ab}$,
$C_{ab}$ is a group with unit $y$ and multiplication
$xz=\Gamma(x,a,y,b,z)$. 
\item The map $\Gamma$ is symmetric under the permutation $(15)(24)$
(reversal of arguments):
\[
\Gamma(x,a,y,b,z)=\Gamma(z,b,y,a,x)
\]
In other words, $U_{ab}$ is the opposite torsor of $U_{ba}$
(same set with reversed product). In particular,
the torsor $U_a:=U_{aa}$ is commutative.
\item The commutative torsor $U_a$ is the underlying additive torsor
of an affine space:
for any $a \in \cX$, $U_a$ is an affine space over $\bK$, with
additive structure given by
\[
x+_y z = \Gamma(x,a,y,a,z),
\]
(sum of $x$ and $z$ with respect to the origin $y$),
and action of scalars given by
\[
sy + (1-s)x = \Pi_s(x,y)
\]
(multiplication of $y$ by $s$ with respect to the origin $x$).
\end{enumerate}
\end{theorem}

\begin{proof}
(i)
Let us show first that $C_{ab}$ is stable under the ternary map
$(xyz)$. Let $x,y,z \in C_{ab}$ and consider the bijective linear
map $g:=M_{xabz}$. We show that $g(y) \in C_{ab}$.
By equation \eqnref{diagonalv}, we have the ``diagonal values''
$M_{xabz}(a)=b$ and $M_{xabz}(b)=a$.
Thus, if $y$ is complementary to $a$ and $b$,
$g(y)$ is complementary both to
$g(a) = b$ and to $g(b) = a$, which means that $g(y)\in C_{ab}$.

The associativity follows immediately from the ``fundamental
relation'' (Proposition \prpref{1.1}(ii)):
\[
(xv(yuz))=L_{xavb}R_{aubz} (y) =
R_{aubz} L_{xavb}(y)  = ((xvy)uz),
\]
and the idempotent laws from
\[
(xxy)=L_{xaxb}(y)=1 (y)=y, \quad
(yxx)=R_{axbx}(y)=1 (y)=y.
\]
Thus $C_{ab}$ is a torsor.

(ii)
This has already been shown in the preceding section.

(iii)
The set $C_a$ is the space of complements of $a$.
It is well-known that this is an affine space over $\bK$.
Let us recall how this affine structure is
defined (see, \textit{e.g.}, \cite{Be04}):
elements $v \in C_a$ are in one-to-one correspondence with projectors
of the form $P^a_v$. Then, for $u,v,w \in U_a$, the structure map
$(u,v,w) \mapsto u+_v w$ in the affine space $C_a$ is given by associating
to $(u,v,w)$ the point corresponding to the projector $P_u^a - P_v^a + P_w^a$,
and the structure map $(v,w) \mapsto r \cdot_v w = (1-r)v+rw$ by associating
to $(v,w)$ the point corresponding to the projector
$r P^a_u + (1-r)P^a_v$.
Now write $y=P_y^a(W)$; then we have
\begin{align*}
\Gamma(x,a,y,a,z) &= (P_x^a - P^z_a)(y) \\
&= (P_x^a - 1 + P_z^a)P_y^a(W) \\
&= (P_x^a - P_y^a + P_z^a) (W),
\end{align*}
and
\[
\Pi_s(x,a,y) =
((1-s)P^a_x + s 1)(z) =
((1-s)P^a_x + s P^a_z)(W),
\]
proving that (iii) describes the usual affine structure of $C_a$.
\end{proof}

\subsubsection*{Homomorphisms}
We think of the maps $\Gamma: D(\Gamma)  \to \cX$ and
$\Pi_r:D(\Pi_r) \to \cX$ as quintary, resp.\ ternary
``product maps'' defined on (parts of) direct products
$\cX^5$, resp.\ $\cX^3$.
Thus we have basic categorical notions just as for groups, rings, modules
etc.:  \emph{homomorphisms} are  maps
$g:\cX \to \cY$
preserving transversality ($x \top y$ implies $g(x) \top g(y)$) and
such that, for all $5$-tuples in $D(\Gamma)$, resp.\ triples in $D(\Pi_r)$,
\begin{align*}
g \left( \Gamma(u,c,v,d,w) \right) &=
\Gamma \left( g(u),g(c),g(v),g(d),g(w)\right)\,, \\
g \left( \Pi_r(u,c,v) \right) &=
\Pi_r \left( g(u),g(c),g(v)\right) \,.
\end{align*}
Essentially, this means that all restrictions of $g$,
\[
U_{ab} \to U_{g(a),g(b)}, \qquad U_a \to U_{g(a)},
\]
are usual homomorphisms (of torsors, resp.\ of affine spaces).
We may summarize this by saying
 that $g$ is ``locally linear'' and ``compatible with all
local group structures''.

\begin{theorem}
\thmlabel{local}
Assume $x,y,z \in U_{ab}$. Then the operators
\[
M_{xabz}:\cX \to \cX, \quad
L_{xayb}:\cX \to \cX, \quad
R_{aybz}:\cX \to \cX
\]
are automorphisms of the geometry $(\cX,\Gamma,\Pi_r)$, and
the groups $(U_{ab},y)$ act on $\cX$ by automorphisms
both from the left and from the right via
\[
(U_{ab},y) \times \cX \to \cX, \quad (x,z) \mapsto
L_{xayb}(z)=\Gamma(x,a,y,b,z),
\]
respectively
\[
\cX \times (U_{ab},y) \to \cX, \quad (x,z) \mapsto
R_{aybz}(x)=\Gamma(x,a,y,b,z).
\]
For fixed $(a,y,b)$,
the left and right  actions commute.
\end{theorem}

\begin{proof}
The construction of the product map $\Gamma$ is ``natural'' in the
sense that all elements of $\Gl_\bB(W)$ (acting
from the left on $W$, commuting with the right $\bB$-module
structure) act by automorphisms of $(\cX,\Gamma)$,
just by ordinary push-forward of sets.
This follows immediately from the relation
$g \circ P^a_x = P^{g(a)}_{g(x)} \circ g$.
In particular, the invertible linear operators $M_{xabz}$,
$L_{xayb}$ and $R_{aybz}$ induce automorphisms of $(\cX,\Gamma)$.

Now fix $y \in U_{ab}$ and consider it as the unit in the group $(U_{ab},y)$.
The claim on the left action amounts to the identities
$L_{yayb} = \id$ (which we already know) and, for all
$x,x' \in U_{ab}$ and all $z \in \cX$,
\[
\Gamma(x,a,y,b,\Gamma(x',a,y,b,z))= \Gamma(\Gamma(x,a,y,b,x'),a,y,b,z).
\]
First, note that, if $z$ is ``sufficiently nice'', i.e., such that
the fundamental relation (Proposition \prpref{1.1}(ii)) applies, then this holds
indeed. We will show in Chapter 2 that the identity in question holds
very generally, and this will prove our claim.
Therefore we leave it as a (slighly lengthy) exercise to the interested
reader to prove the claim in the present framework.
The claims concerning the right action are proved in the same way,
and the fact that both actions commute
is precisely the content of the fundamental relation (Proposition \prpref{1.1}(ii))
\end{proof}

\subsubsection*{Inner automorphisms}
We call automorphisms of the geometry defined by the preceding
theorem \emph{inner automorphisms}, and the group generated by
them the \emph{inner automorphism group}.
Note that middle multiplications $M_{xabz}$ are honest automorphisms
of the geometry $(\cX,\Gamma)$, although they are
\emph{anti}-automorphisms of the torsor $U_{ab}$; this is due to the
fact that they exchange $a$ and $b$.
On the other hand, $L_{xayb}$ and $R_{aybz}$  are
automorphisms  of the whole geometry \emph{and}
of $U_{ab}$.

Note also that the action of the  groups $U_{ab}$
is of course very far from being regular on its orbits, except
on $U_{ab}$ itself. For instance, $a$ and $b$ are fixed points of these
actions, since $\Gamma(x,a,y,b,b)=b$ and $\Gamma(x,a,y,b,a)=a$.

Finally, the statements of the preceding two theorems amount to
certain algebraic identities for the multiplication map $\Gamma$.
This will be taken up in Chapter 2, where we will not have to worry
about domains of definition.

\subsection{Affine picture of the torsor $U_{ab}$}
\sublabel{affine}

It is useful to have ``explicit formulas'' for our map $\Gamma$.
Such formulas can be obtained by introducing ``coordinates''
on $\cX$ in the following way (see  \cite{Be04}).
First of all, choose a base point $(o^+,o^-)$ and consider the
pair geometry $(\cX^+,\cX^-)$, where $\cX^\pm$ is the space of
all submodules isomorphic to $o^\pm$ and having a complement isomorphic
to $o^\mp$.
We identify $\cX^+$ with injections $x:o^+ \to W$ of $\bB$ right-modules,
modulo equivalence under the action of the group $G:=\Gl(o^+)$
($x \cong x \circ g$, where $g$ acts on $o^+$
on the left), and $\cX^-$ with $\bB$-linear
surjections $a:W \to o^+$ (modulo equivalence
$a \cong g \circ a$ for $g \in G$).
Equivalence classes are denoted by $[x]$, resp.\ $[a]$.

\begin{proposition}
\prplabel{formulae}
The following explicit formulae hold for $x,y,z \in \cX^+$, $a,b \in \cX^-$.
\begin{enumerate}[label=\roman*\emph{)},leftmargin=*]
\item if $x \top a$ and $z \top b$ (middle multiplication), then
\begin{equation}
\eqnlabel{form1}
\Gamma\Big( [x],[a],[y],[b],[z]\Big) =
\Big\lbrack x(ax)\inv ay -y + z (bz)\inv by \Big\rbrack\,,
\end{equation}
\item if $a \top x$ and $b \top y$ (left multiplication), then
\begin{equation}
\eqnlabel{form2}
\Gamma\Big( [x],[a],[y],[b],[z]\Big) =
\Big\lbrack x(ax)\inv ay(by)\inv(bz)  -y(by)\inv(bz) + z  \Big\rbrack\,,
\end{equation}
\item  if $a \top y$ and $b \top z$ (right multiplication), then
\begin{equation}
\eqnlabel{form3}
\Gamma\Big( [x],[a],[y],[b],[z]\Big) =
\Big\lbrack x- y (ay)\inv ax + z (bz)\inv za (ay)\inv ax  \Big\rbrack\,.
\end{equation}
\end{enumerate}
\end{proposition}

\begin{proof}
The right hand side of \eqnref{form1} is a well-defined element of $\cX$, as is seen
be replacing $x$ by $x \circ g$, resp.\ $y$ by $y \circ g$, $z$ by
$z  \circ g$ and $a$ by $g \circ a$, $b$ by $g \circ b$.
Note that $[x]$ and $[a]$ are transversal if and only if
$ax:o^+ \to o^+$ is invertible.
Now,  the operator
\[
x(ax)\inv a: W \to o^+ \to W
\]
has kernel $a$ and image $x$ and is idempotent, therefore it
is  $P^a_x$. Similarly, we see that
$z (bz)\inv b$ is $P_z^b$, and hence the right hand side is
induced by the operator
$P^a_x  - 1 + P^b_z =M_{xabz}$.
Similarly, we see that the right hand side of \eqnref{form2}
is induced by the linear operator
\[
P_x^a P^b_y - P_y^b + 1 = (P_x^a - 1)P^b_y + 1 = 1 - P_a^x P_y^b =
L_{axby}
\]
and the one of \eqnref{form3} by
$
1-P_y^a + P_z^b P^a_y=1+(P^b_z -1)P^a_y = 1 - P^z_b P^a_y =
R_{zbya}.
$
\end{proof}

As usual in projective geometry, the projective formulas from the preceding
result may be affinely re-written:  if $y \top b$, we may
affinize by taking $([y],[b])$ as base point $(o^+,o^-)$:
we write $W=o^- \oplus o^+$; then  injections $x:o^+ \to W$,
$z:o^+ \to W$ that are transversal to the first factor
can be identified with column vectors (by normalizing the second component
to be the identity operator on $o^+$)
\[
x=\begin{pmatrix} X \cr 1\cr\end{pmatrix} , \qquad
z=\begin{pmatrix} Z \cr 1\cr\end{pmatrix}
\]
(columns with $X,Z \in \Hom(o^+,o^-)$).
In other terms, $x$ and $z$ are graphs of linear operators $X,Z : o^+ \to o^-$.
Surjections $a:W \to o^+$ that are transversal to the second factor
correspond to row vectors $(A,1)$ (row with $A \in \Hom(o^-,o^+)$).
Note, however, that the kernel of $(A,1)$ is determined by the
condition $Au + v=0$, i.e., $v = -Au$, and hence $a$ is the graph of
$-A:o^- \to o^+$. Therefore we write $a = (-A,1)$.
The base point $y = o^+$ is the column $(0,1)^t$, and the base point $b=o^-$
is the row $(0,1)$.
Since $ax=(-A,1)(X,1)^t = 1 - AX$, $a$ and $x$ are transversal iff
$1 - AX:o^+ \to o^+$ is an invertible operator (in Jordan theoretic language:
the pair $(X,A)$ is \emph{quasi-invertible}, cf.\ \cite{Lo75}).
Using this, any of the three formulas from the preceding proposition leads to
the ``affine picture'':
\[
\Gamma(x,a,y,b,z)=
\begin{bmatrix}
\begin{pmatrix} X \cr 1 \end{pmatrix} (1-AX)\inv -
\begin{pmatrix} 0 \cr 1 \end{pmatrix} +
\begin{pmatrix} Z \cr 1 \end{pmatrix}
\end{bmatrix}
=
\begin{bmatrix}
\begin{pmatrix} - ZAX + X + Z \cr 1 \end{pmatrix}
\end{bmatrix}\,.
\]
Finally, identifying $x$ with $X$, $y$ with $Y$ and so on, we may
write
\[
\Gamma(X,A,O^+,O^-,Z) = X - ZAX +  Z\,.
\]
This formula  is interesting in many respects:
it is affine in all three variables, and the product $ZAX$ from
the associative pair $\bigl(\Hom(o^+,o^-),\Hom(o^-,o^+)\bigr)$ shows up.
We will give conceptual explanations of these facts later on.
Also, it is an easy exercise to check directly that
$(X,Z) \mapsto X - ZAX +  Z$ defines an associative product on
$\Hom(o^+,o^-)$ and induces a
 group structure on the
set of elements $X$ such that $1-AX$ is invertible.

\subsubsection*{Other ``rational'' formulas}
More generally, having fixed $(o^+,o^-)$, we may write $a,b$ as row-,
and $x,y,z$ as column vectors, and then we get the general formula
\begin{multline*}
\Gamma(X,A,Y,B,Z) = \\
\begin{bmatrix}
\begin{pmatrix} X \cr 1 \end{pmatrix} (1-AX)\inv(1-AY) -
\begin{pmatrix} Y \cr 1 \end{pmatrix} +
\begin{pmatrix} Z \cr 1 \end{pmatrix} (1-BZ)\inv (1-BY)
\end{bmatrix}\,,
\end{multline*}
which is (the class of) a vector with second component (``denominator'')
\[
D:=(1-AX)\inv(1-AY) - 1 + (1-BZ)\inv (1-BY),
\]
and first component (``numerator'')
\[
N:= X (1-AX)\inv(1-AY) - Y + Z (1-BZ)\inv (1-BY),
\]
so that the affine formula is $\Gamma(X,A,Y,B,Z)=ND\inv$.
Besides the above choice ($Y=O^+$, $B=O^-$), another reasonable choice is
just $B=O^-$, leading to
\[
\Gamma(X,A,Y,O^-,Z)=X- (Y-Z)(1-AY)\inv (1-AX)\,.
\]
Similarly, for $Y=O^+$ we get formulas that, in case $A=B$, correspond to
well-known Jordan theoretic formulas for the quasi-inverse.
Such formulas show that, if we work in finite dimension over a field,
$\Gamma$ is a rational map in the sense of algebraic geometry, and
if we work in a topological setting over topological fields or rings,
then $\Gamma$ will have smoothness properties similar to the ones
described in \cite{BeNe05}.

\subsubsection*{Case of a geometry of the first kind}
Assume there is a transversal triple, say, $(o^+,e,o^-)$.
We may assume that $e$ is the diagonal in $W=o^- \oplus o^+$.
Take, in the formulas given above, $a=0=(0,1)$,  $b=\infty=(1,0)$, $y=(1,1)^t$,
$ax=(0,1)(X, 1)^t=1$,
$bz=(1,0)(Z,1)^t=Z$,
$ay=1$, $by=1$, so we get
\[
\Gamma(X,0,e,\infty,Z)=
\begin{bmatrix}
\begin{pmatrix} X \cr 1 \cr \end{pmatrix} -
\begin{pmatrix} 1 \cr 1 \cr \end{pmatrix} +
\begin{pmatrix} Z \cr 1 \cr \end{pmatrix} Z\inv
\end{bmatrix} =
\begin{bmatrix}
\begin{pmatrix} X \cr Z\inv \cr \end{pmatrix}
\end{bmatrix} =
\begin{bmatrix}
\begin{pmatrix} XZ \cr 1 \cr \end{pmatrix}
\end{bmatrix}\,,
\]
and hence the affine picture is the algebra
$\End_{\bB}(o^+)$  with  its usual product.
Taking $a=\infty$, $b=0$ gives the opposite of the usual product.
Replacing $e$ by $y = \setof{(v,Yv)}{Y:o^+ \to o^-}$ (graph of an invertible
linear map $Y$), we get the affine picture
\[
\Gamma(X,0,Y,\infty,Z)=
\begin{bmatrix}
\begin{pmatrix} XY\inv Z \cr 1 \cr \end{pmatrix}
\end{bmatrix}\,.
\]

\subsection{Affinization: the transversal case}

If $a$ and $b$ are arbitrary, then in general the torsor
$U_{ab}$ will be empty. Therefore we look at the pair $(U_a,U_b)$.

\begin{theorem}\label{pair!}
For all $a,b \in \cX$, we have
\[
\Gamma(U_a,a,U_b,b,U_a) \subset U_a, \quad
\Gamma(U_b,a,U_a,b,U_b) \subset U_b \, .
\]
In other words,  the maps
\begin{align*}
U_a \times U_b \times U_a \to U_a ; & \quad
(x,y,z) \mapsto (xyz)^+:=L_{xayb}(z)=\Gamma(x,a,y,b,z)\,, \\
U_b \times U_a \times U_b \to U_b ; & \quad
(x,y,z) \mapsto (xyz)^-:=R_{aybz}(x)=\Gamma(x,a,y,b,z)
\end{align*}
are well-defined.  If, moreover, $a \top b$, then both maps
are trilinear, and they form an \emph{associative pair}, i.e., they
satisfy the para-associative law (cf.\ Appendix B)
\[
(xy(uvw)^\pm)^\pm = ((xyu)^\pm vw)^\pm = (x(vuy)^\mp w)^\pm.
\]
\end{theorem}

\begin{proof}
Assume that $x \top a$ and $y \top b$.
By a direct calculation, we will show that $L_{xayb}(U_a) \subset U_a$.
Let us write $L_{xayb}$ in matrix form with respect to the decomposition
$W=a \oplus x$.
The projectors $P^x_a$ and $P^b_y$ can be written
\[
P^x_a=\begin{pmatrix} 1 & 0 \cr 0 & 0 \end{pmatrix}, \quad
P^b_y=\begin{pmatrix} \alpha & \beta \cr \gamma & \delta \end{pmatrix}
\]
whith $\alpha \in \End(a)$, $\beta \in \Hom(x,a)$, etc.
Thus
\[
L_{xayb} = 1 - \begin{pmatrix} 1 & 0 \cr 0 & 0 \end{pmatrix}
\begin{pmatrix} \alpha & \beta \cr \gamma & \delta \end{pmatrix} =
\begin{pmatrix} 1-\alpha & -\beta \cr 0 & 1 \end{pmatrix} .
\]
Let $z \in U_a$; it can be written as the graph
$\{ (Zv,v) | \, v \in x \}$ of a linear operator $Z:x \to a$.
Since
\[
L_{xayb} \begin{pmatrix} Zv \cr v \cr\end{pmatrix}
 = \begin{pmatrix} 1-\alpha & -\beta \cr 0 & 1 \end{pmatrix}
\begin{pmatrix} Zv \cr v \cr\end{pmatrix} =
\begin{pmatrix} (1-\alpha)Zv - \beta v \cr v \cr \end{pmatrix},
\]
$L_{xayb}(z)$ is the graph of the linear operator
$(1-\alpha)Z - \beta : x \to a$, and hence is again transversal to $a$,
so $( \quad )^+$ is well-defined.
By symmetry, it follows that $( \quad )^-$ is well-defined.
Moreover, the calculation shows  that $z \mapsto (xyz)^+$
is affine (we will see later that this map is actually affine with respect to
all three variables, see Corollary \ref{pair!!}).

\smallskip
Now assume that $a \top b$, and
 write $L_{xayb}$ in matrix form with respect to the decomposition
$W=a \oplus b$.
The projectors $P^a_x$ and $P^b_y$ can be written
\[
P^a_x=\begin{pmatrix} 0 & X \cr 0 & 1 \end{pmatrix}, \quad
P^b_y=\begin{pmatrix} 1 & 0 \cr Y & 0 \end{pmatrix}
\]
where $X \in \Hom(b,a)$ and $Y \in \Hom(a,b)$.
We get
\[
L_{xayb} = 1 - \begin{pmatrix} 1-\begin{pmatrix} 0 & X \cr 0 & 1 \end{pmatrix} \end{pmatrix}
\begin{pmatrix} 1 & 0 \cr Y & 0 \end{pmatrix} =
1 - \begin{pmatrix} 1-XY & 0 \cr 0 & 0 \end{pmatrix}
= \begin{pmatrix}  XY & 0   \cr  0 & 1  \end{pmatrix}
\]
and, writing  $z \in U_a$ as a graph $\{ (Zv,v) | v \in b \}$, we get
\[
L_{xayb} \begin{pmatrix} Zv \cr v \end{pmatrix} =
\begin{pmatrix}  XY & 0   \cr  0 & 1  \end{pmatrix}
\begin{pmatrix} Zv \cr v \end{pmatrix} =
 \begin{pmatrix} XYZ v  \cr v \end{pmatrix} ,
\]
hence $L_{xayb}(z)$ is the graph of $XYZ:b \to a$. Thus,
with $V^+=U_a\cong \Hom(b,a)$, $V^-=U_b \cong \Hom(a,b)$,
the first ternary map is given by
\[
V^+ \times V^- \times V^+ \to V^+, \quad (X,Y,Z) \mapsto XYZ.
\]
Similarly, one shows that the second ternary map is
given by
\[
V^- \times V^+ \times V^- \to V^-, \quad (X,Y,Z) \mapsto ZYX.
\]
This pair of maps is the prototype of an associative pair (see Appendix B).
\end{proof}

At this stage, the appearance of the trilinear expression $ZYX$, resp.\
$ZAX$, both in the affine pictures of the map from the preceding theorem
and in the preceding section, related by the identity
\begin{equation}
X - (X - ZAX +  Z) + Z =  ZAX , \label{coinc}
\end{equation}
looks like a pure coincidence.
A conceptual explanation will be given in Chapter 3 (Lemma \ref{coinc'}).

\section{Grassmannian semitorsors}
\seclabel{semitorsor}

In this chapter we extend the definition of the product map
$\Gamma$ onto all of $\cX^5$, and we
show that  the most important algebraic identities extend also.
We use notation and general notions explained in the first
section of the preceding chapter.

\subsection{Composition of relations}

Recall that, if $A,B,C,\ldots$ are any sets,
we can \emph{compose
relations}: for subsets
$x \subset A \times B$, $y \subset B \times C$,
\[
y \circ x := yx :=  \setof{(u,w) \in A \times C}
{\exists v \in B: \, (u,v) \in x, (v,w) \in y}\,.
\]
Composition is associative: both $(z \circ y) \circ x$ and
$z \circ (y \circ x)$ are equal to
\begin{equation}
z \circ y \circ x = \setof{(u,w) \in A \times D}
{\exists (v_1,v_2) \in y: \,
(u,v_1) \in x, (v_2,w) \in z} \,.
\label{assoc}
\end{equation}
If $x$ and $y$ are graphs of
maps $X$, resp.\ $Y$ ($v=Xu$, $w=Yv$) then $y \circ x$
is the graph of $YX$   ($w=Yv=YXu$).
The \emph{reverse relation} of $x$ is
$$
x\inv := \setof{(w,v) \in B \times A}{(v,w) \in x}.
$$
We have $(yx)\inv = x\inv y\inv$, and
if $x$ is the graph of a bijective map, then
$x\inv$ is the graph of its inverse map.
For $x,y,z \subset A \times B$, we get another relation between $A$ and $B$ by
$zy^{-1}x $. Obviously, this ternary composition satisfies the para-associative law, and hence
relations between $A$ and $B$ form
a semitorsor. Letting $W:=A \times B$, we have the explicit formula
\begin{align*}
zy^{-1}x &=
\Bigsetof{\omega = (\alpha',\beta') \in W}
{\begin{array}{c}
\exists \eta=(\alpha'',\beta'') \in y:
 \\
(\alpha',\beta'') \in x, (\alpha'',\beta') \in z
\end{array}}\,
\\
& =
\Bigsetof{\omega  \in W}
{\begin{array}{c}
\exists \alpha',\alpha'' \in A, \exists \beta',\beta'' \in B, \exists
\eta \in y, \exists
\xi \in x, \exists \zeta \in z:
 \\
\omega = (\alpha',\beta'),
\eta = (\alpha'',\beta''),
\xi=(\alpha',\beta''),
\zeta = (\alpha'',\beta')
\end{array}}\,
\end{align*}

\subsection{Composition of linear relations}

Now assume that $A,B,C,\ldots$
are {\em linear} spaces over $\bB$ (i.e., right modules)
and that all relations are \emph{linear
relations} (i.e., submodules of $A \oplus B$, etc.).
Then $zy^{-1}x$ is again a linear relation.
Identifying $A$ with the first and $B$ with the second factor
in $W:=A \oplus B$, the description of $zy^{-1}x$ given above can
be rewritten, by introducing the new variables
$\alpha:=\alpha' - \alpha''$,
$\beta:=\beta' - \beta''$,
\begin{align*}
zy^{-1}x &=
\Bigsetof{\omega  \in W}
{\begin{array}{c}
\exists \alpha',\alpha'' \in a, \exists \beta',\beta'' \in b, \exists
\eta \in y,
\exists \xi \in x, \exists \zeta \in z:
 \\
\omega = \alpha' + \beta',
\eta = \alpha'' + \beta'',
\xi=\alpha' + \beta'',
\zeta = \alpha'' + \beta'
\end{array}}\,
\\
&=
\Bigsetof{\omega  \in W}
{\begin{array}{c}
\exists \alpha',\alpha \in a, \exists \beta',\beta \in b, \exists \eta \in y,
\exists \xi \in x, \exists \zeta \in z:
 \\
\omega = \alpha' + \beta',
\eta = \omega - \alpha -  \beta,
\xi=\omega - \beta,
\zeta =  \omega - \alpha
\end{array}}\, .
\end{align*}
In order to stress that the product $xy^{-1}z$ depends also on
$A$ and $B$,  we will henceforth use lowercase letters $a$ and $b$
and write $W=a \oplus b$.

\begin{lemma}
Assume $W=a \oplus b$ and let $x,y,z \in \Gras_\bB(W)$. Then
\[
zy\inv x =
\Bigsetof{\omega \in W}
{\begin{array}{c}
\exists \xi \in x,
\exists \alpha \in a,
\exists \eta \in y,
\exists \beta \in b,
\exists \zeta \in z : \\
\omega = \zeta + \alpha
= \alpha + \eta + \beta
= \xi + \beta
\end{array}}\,.
\]
\end{lemma}

\begin{proof}
Since $W = a \oplus b$, the first condition ($\exists
\alpha' \in a, \beta' \in b$: $\omega = \alpha' + \beta'$)
in the preceding description is always satisfied and can hence
be omitted in the description of $zy\inv x$.
Replacing $\alpha$ by $-\alpha$ and $\beta$ by $-\beta$, the claim
follows.
\end{proof}

\subsection{The extended product map}
\sublabel{extended}

Motivated by the considerations from the preceding section, we now
\emph{define} the product map
$\Gamma:\cX^5 \to \cX$ for \emph{all} $5$-tuples of the Grassmannian
$\cX = \Gras_\bB(W)$ by
\[
\Gamma(x,a,y,b,z)  :=
\Bigsetof{\omega \in W}
{\begin{array}{c}
\exists \xi \in x,
\exists \alpha \in a,
\exists \eta \in y,
\exists \beta \in b,
\exists \zeta \in z : \\
\omega = \zeta + \alpha
= \alpha + \eta + \beta
= \xi + \beta
\end{array}}\,.
\]
We will show, among other things, that this notation is
in keeping with the one introduced in the preceding chapter.
Firstly, however, we collect various equivalent formulas for
$\Gamma$. The three conditions
\begin{equation}
\begin{matrix}
\omega & = & \eta + \alpha + \beta \cr
\omega & = & \beta + \xi \cr
\omega & = & \alpha + \zeta
\end{matrix}
\label{(2.1)}
\end{equation}
can be re-written in various ways. For instance, subtracting
the last two equations from the first one  we get the equivalent
conditions
\begin{equation}
\omega  =  - \eta + \xi + \zeta, \quad
\omega  =  \beta + \xi, \quad
\omega =  \alpha + \zeta
\label{(2.2)}
\end{equation}
and hence, replacing $\eta$ by $-\eta$, we get
\[
\Gamma(x,a,y,b,z)=
\Bigsetof{\omega \in W}
{\begin{array}{c}
\exists \xi \in x,
\exists \alpha \in a,
\exists \eta \in y,
\exists \beta \in b,
\exists \zeta \in z : \\
\omega = \zeta + \alpha
= \xi + \eta + \zeta
= \xi + \beta
\end{array}}\, .
\]
Next, letting $\alpha'=-\alpha$ and $\beta'=-\beta$,
conditions (\ref{(2.1)}) are equivalent to
\begin{equation}
\eta = \omega + \alpha' + \beta', \quad
\zeta = \omega + \alpha',\quad
\xi = \omega + \beta' ,
\label{(2.3)}
\end{equation}
and hence we get
\begin{align*}
\Gamma(x,a,y,b,z) &=
\Bigsetof{\omega \in W}
{\begin{array}{c}
\exists \xi \in x,
\exists \alpha \in a,
\exists \eta \in y,
\exists \beta \in b,
\exists \zeta \in z : \\
\eta = \omega + \alpha + \beta,
\zeta = \omega + \alpha,
\xi = \omega + \beta
\end{array}}
\cr
&=
\Bigsetof{ \omega \in W}{\exists \beta \in b, \exists \alpha \in a :
\omega + \alpha \in z,
\omega + \alpha + \beta  \in  y ,
\omega + \beta \in  x} .
\end{align*}
The following lemma now follows by  straightforward
changes of variables:

\begin{lemma}
\lemlabel{gamma-equivs}
For all $x,a,y,b,z \in \cX$,
\begin{align*}
\Gamma(x,a,y,b,z) &=
\Bigsetof{ \omega \in W}{\exists \xi \in x, \exists \zeta \in z :
\zeta + \omega  \in   a,
\zeta + \omega+ \xi  \in  y ,
\omega + \xi \in  b}  \\
&=
\Bigsetof{ \omega \in W}{\exists \xi \in x, \exists \alpha \in a :
\omega - \alpha \in z,
\xi - \alpha  \in  y ,
\omega - \xi \in  b}  \\
&=
\Bigsetof{ \omega \in W}{\exists \beta \in b, \exists \zeta \in z :
\zeta - \omega  \in   a,
\zeta - \beta  \in  y ,
\omega - \beta \in x}  \\
&=
\Bigsetof{\omega \in W}{\exists \eta \in y, \exists \beta \in b :
\omega - \eta - \beta  \in   a,
\beta + \eta \in z ,
\omega - \beta \in x}  \\
&=
\Bigsetof{ \omega \in W}{\exists \eta \in y, \exists \zeta \in z :
\omega  + \zeta  \in   a,
\zeta + \eta \in b ,
\omega + \zeta + \eta \in x}
\end{align*}
\end{lemma}

We refer to the descriptions of the lemma as the ``$(x,z)$-'',
``$(x,a)$-description'', and so on.
The $(a,b)$-description is particularly useful for the proof of the
theorem below.
One may note that the only pairs of variables that cannot be used for
such a description
are $(a,z)$ and $(x, b)$, and that the signs in the terms appearing in these
descriptions can be chosen positive if the pair is ``homogeneous''
(a subpair of $(x,y,z)$ or of $(a,b)$), whereas for ``mixed'' pairs we
cannot get rid of signs.

\begin{theorem}
\label{MainTheorem}
The map $\Gamma: \cX^5 \to \cX$ extends the product map
defined in the preceding chapter, and has the following properties:
\begin{enumerate}[label=\emph{(}\arabic*\emph{)},leftmargin=*]
\item
It is symmetric under the Klein $4$-group:
\begin{align*}
\Gamma(x,a,y,b,z) &= \Gamma(z,b,y,a,x)\,, \tag{a} \\
\Gamma(x,a,y,b,z) &= \Gamma(a,x,y,z,b)\,. \tag{b}
\end{align*}
\item
For any pair $(a,b) \in \cX^2$, the product
$(xyz) := \Gamma(x,a,y,b,z)$ on $\cX^3$ satisfies the properties
of a semitorsor, that is,
\[
\Gamma\bigl(x,a,u,b,\Gamma(y,a,v,b,z)\bigr) =
\Gamma\bigl(x,a,\Gamma(v,a,y,b,u),b,z\bigr) =
\Gamma\bigl(\Gamma(x,a,u,b,y),a,v,b,z\bigr)\,.
\]
We will write $\cX_{ab}$ for $\cX$ equipped with this semitorsor
structure. Then the semitorsor $\cX_{ba}$ is the opposite semitorsor
of $\cX_{ab}$; in particular, $\cX_{aa}$ is a commutative semitorsor, for
any $a$.
\end{enumerate}
\end{theorem}

\begin{proof}
(1) The symmetry relation (a) is obvious from the definition of $\Gamma$.
Exchanging  $x$ and $a$  amounts in the
$(x,a)$-description
to exchanging simultaneously $z$ and $b$, hence the symmetry relation (b) follows.

For (2), we use the $(a,b)$-description: on the one hand,

\smallskip
\noindent$\Gamma\big(x,a,u,b,\Gamma(y,a,v,b,z)\big) =$
\begin{align*}
&= \Bigsetof{ \omega \in W}{
\begin{array}{c}
\exists \alpha \in a, \exists \beta \in b : \\
\omega + \alpha \in \Gamma(y,a,v,b,z), \omega + \alpha + \beta \in u,
\omega + \beta \in x
\end{array} } \\
&= \Bigsetof{ \omega \in W}{
\begin{array}{c}
\exists \alpha \in a, \exists \beta \in b,
\exists \alpha' \in a, \exists \beta' \in b : \\
\omega + \alpha + \beta \in u,
\omega + \beta \in x,
\omega + \alpha + \alpha' \in z, \\
\omega + \alpha + \alpha' + \beta' \in v,
\omega + \alpha + \beta' \in y
\end{array} }
\end{align*}
On the other hand,

\smallskip
\noindent$\Gamma\big(x,a,\Gamma(u,b,y,a,v),b,z\big) =$
\begin{align*}
&= \Bigsetof{ \omega \in W}{
\begin{array}{c}
\exists \alpha'' \in a, \exists \beta'' \in b : \\
\omega + \alpha'' \in z, \omega + \alpha'' + \beta'' \in \Gamma(u,b,y,a,v),
\omega + \beta'' \in x
\end{array} } \\
&= \Bigsetof{ \omega \in W}{
\begin{array}{c}
\exists \alpha'' \in a, \exists \beta'' \in b,
\exists \alpha''' \in a, \exists \beta''' \in b : \\
\omega + \alpha'' \in z,\omega + \beta'' \in x,
\omega + \alpha'' + \beta'' + \alpha''' \in u, \\
\omega + \alpha'' + \beta'' + \beta''' \in v,
\omega + \alpha ''+ \beta'' + \alpha''' + \beta''' \in y
\end{array} }
\end{align*}
Via the change of variables
$\alpha''=\alpha + \alpha'$,
$\alpha''' = \alpha'$,
$\beta''=\beta$,
$\beta'''=\beta$, we see that these two subspaces of $W$ are the same.
The remaining equality is equivalent to the one just proved via
the symmetry relation (a).

Next, we show that the new map $\Gamma$ coincides with the old one
on $D(\Gamma)$.
Let us assume that $(x,a,y,b,z) \in D_L$, so $x \top a$ and $y \top b$.
We use the $(y,b)$-description and let
$\zeta:=\eta + \beta$, whence $\eta = P_y^b \zeta$ and
$\beta = P_b^y \zeta$. We get
\begin{align*}
\Gamma(x,a,y,b,z) &=
\Bigsetof{ \omega \in W}{\exists \eta \in y, \exists \beta \in b :
\omega - \eta - \beta  \in   a,
\beta + \eta \in z ,
\omega - \beta \in x}  \\
&=
\Bigsetof{ \omega \in W}{\exists \zeta \in z :
\omega - P_b^y(\zeta) \in x,
\omega - \zeta \in  a}  \\
&=
\Bigsetof{ \omega \in W}{\exists \zeta \in z :
P^a_x(\omega - \zeta)=0,
P^a_x(\omega - P_b^y(\zeta)) = \omega - P_b^y(\zeta)}  \\
&=
\Bigsetof{ \omega \in W}{\exists \zeta \in z :
P_x^a \zeta = P_x^a \omega,
\omega = P_b^y \zeta + P^a_x \omega - P_x^a P_b^y \zeta } \\
&=
\Bigsetof{ \omega \in W}{\exists \zeta \in z :
\omega = (P_b^y  + P^a_x  - P_x^a P_b^y) \zeta}
\end{align*}
and a straightforward calculation shows that
\[
P_b^y  + P^a_x  - P_x^a P_b^y = 1 - P_a^x P_y^b = L_{xayb}
\]
so that $\Gamma(x,a,y,b,z) =L_{xayb}(z)$.
This proves that the old and new definitions of $\Gamma$ coincide on
$D_L$, and hence also on $D_R$ by the symmetry relation.
Now we show that the new map $\Gamma$ coincides with the old one on
$D_M$: assume $a \top x$ and $b \top z$ and use the $(x,z)$-description;
let $\eta:=\zeta - \omega + \xi$ and observe that
$P^a_x \eta = P^a_x \xi = \xi$ (since $\zeta - \omega \in a$), and
similarly $P^b_z \eta = \zeta$, whence
$\omega = \zeta - \eta + \xi = (P_z^b - 1 + P_x^a) \eta$,
and thus
\begin{align*}
\Gamma(x,a,y,b,z) &=
\Bigsetof{ \omega \in W}{\exists \xi \in x, \exists \zeta \in z :
\zeta - \omega  \in   a,
\zeta - \omega+ \xi  \in  y ,
\omega- \xi \in  b}  \\
&=
\Bigsetof{ \omega \in W}{\exists \eta \in y :
\omega =(P_z^b - 1 + P_x^a) \eta}\,,
\end{align*}
that is, $\omega = - M_{xabz} \eta$,
and hence $\Gamma(x,a,y,b,z)=M_{xabz}(y)$.
\end{proof}

\subsection{Diagonal values}

We call {\em diagonal values} the values taken by $\Gamma$ on the
subset of $\cX^5$ where at least two of the
five variables $x,a,y,b,z$ take the same value.
There are two different kinds of behavior on such diagonals:
for the diagonal $a=b$ (or, equivalently, $x=z$), we still have
a rich algebraic theory which is equivalent to the {\em Jordan part}
of our associative products; this topic is left for subsequent work
(cf.\ Chapter 4). The three remaining diagonals ($x=y$, resp.\
$a=z$, resp.\ $b=z$) have an entirely different behavior: the algebraic
operation $\Gamma$ restricts in these cases to {\em lattice theoretic
operations}, that is, can be expressed by intersections and sums
of subspaces. We will use the
lattice theoretic notation $x \land y = x \cap y$ and
$x \lor y = x + y$. It is remarkable that two important aspects
of projective geometry (the lattice theoretic and the Jordan theoretic)
arise as a sort of ``contraction'' of the full map $\Gamma$, or, put
differently, that they have a common ``deformation'', given by
$\Gamma$.

\begin{theorem}
\label{lattice}
The map $\Gamma: \cX^5 \to \cX$  takes the following diagonal values:
\begin{enumerate}[label=\emph{(}\arabic*\emph{)},leftmargin=*]
\item values on the ``diagonal $x=y$'': for all $(x,a,b,z) \in \cX^4$,
$$
\Gamma(x,a,x,b,z)= 
(z \lor (x \land a)) \land (b \lor x).
$$
In particular,
we get the following ``subdiagonal values'':
for all $x,a,y,b,z$,
\begin{enumerate}[label=\emph{(}\roman*\emph{)},leftmargin=*]
\item
subdiagonal $x=y=z$:  $\Gamma(x,a,x,b,x) = x$
(law  $(xxx)=x$ in $\cX_{ab}$),
\item
subdiagonal $x=y=a$:
 $\Gamma(x,x,x,b,z) = (z \lor x) \land (b \lor x)$
\item
subdiagonal $x=y=a$ and $b=z$:
 $\Gamma(x,x,x,z,z) = z \lor x $
\item
subdiagonal $x=y=b$:
$\Gamma(x,a,x,x,z)=
(z \lor (x \land a)) \land x$.
\item
subdiagonal $x=y=b$ and $a=z$:
 $\Gamma(x,a,x,x,a) = a \land x$
\item
subdiagonal $x=y$, $a=z$:
$\Gamma(x,a,x,b,a) = a \land (b \lor x)$
\item
subdiagonal $x=y$, $a=b$: $\Gamma(x,a,x,a,z)=(z \lor (x \land a)) \land (x \lor a)$
\item
subdiagonal $x=y$, $z=b$: $\Gamma(x,a,x,z,z)=z \lor (x \land a)$
\end{enumerate}

\item
diagonal $a=z$: for all $(x,a,y,b) \in  \cX^4$,
$$
\Gamma(x,a,y,b,a)= a \land (b \lor (x \land (y \lor a)))
$$
In particular, on the subdiagonal $x=z=b$, we have, for all
$x,a,y \in \cX$,
$$
\Gamma(x,a,y,x,a)=a \land x.
$$

\item
diagonal $b=z$: for all $(x,a,y,b) \in \cX^4$,
$$
\Gamma(x,a,y,b,b)= b \lor (a \land (x \lor (y \land b)))
$$
In particular, on the subdiagonal $b=z$, $x=a$, we have, for all
$a,y,b \in \cX$,
$$
\Gamma(a,a,y,b,b)=b \lor a ,
$$
and on $a=b=az$:  for all $x,a,y$, $\Gamma(x,a,y,a,a)=a$.
\end{enumerate}
\end{theorem}

\begin{proof}
In the following proof, in order to avoid unnecessary repetitions,
 it is always understood that $\alpha \in a$, $\xi \in x$, $\beta \in b$,
$\eta \in y$, $\zeta \in z$.
In all three items, the determination
 of the ``subdiagonal values'' is a  straightforward
consequence, using  the
absorption laws $u \lor (u \land v)=u$, $u \land (u \lor v)=u$.

Now we prove (1) (diagonal $x=y$).
Let $\omega \in \Gamma(x,a,x,b,z)$, then
$\omega = \xi + \beta$, hence $\omega \in (x \lor b)$,
and
$\omega = \eta + \xi + \zeta$ with
$v:=\omega - \zeta = \eta + \xi \in x$ (since $x=y$).
On the other hand,
$v=\omega - \zeta = \alpha \in a$, whence
$\omega = v + \zeta$ with $v \in (x \land a)$, proving one inclusion.

Conversely, let $\omega \in (z \lor (x \land a)) \land (b \lor x)$.
Then
$\omega = \beta + \xi = \alpha + \zeta$
with $\alpha \in (x \land a)$.
Let $\eta := \xi - \alpha$. Then $\eta \in x$, and
$\omega = \xi + \beta = \eta + \alpha + \beta$, hence
$\omega \in\Gamma(x,a,x,b,z)$.

Next we prove
(2) (diagonal $a=z$).
Let $\omega \in \Gamma(x,a,y,b,a)$, then
$\omega = \zeta + \alpha$ with $\zeta,\alpha \in z=a$, whence
$\omega \in a$.
Moreover, $\omega = \xi + \beta = \eta + \alpha + \beta$, with
$\eta + \alpha = \xi \in x$ and
$\eta + \alpha \in y \lor a$, whence
$\omega \in b \lor (x \land (y \lor a))$.

Conversely, let $\omega \in a \land (b \lor (x \land (y \lor a)))$.
Then $\omega \in b \lor (x \land (y \lor a))$, that is,
$\omega = \beta + (\eta + \alpha)$ with $\xi:=\eta + \alpha \in x$.
Letting $\zeta := \omega - \alpha \in a$ (here we use that $\omega \in a$),
we have
$\omega = \zeta + \alpha$, proving that
$\omega \in \Gamma(x,a,y,b,a)$.

The proof for (3) (diagonal $z=b$) is ``dual'' to the preceding one
and will be left to the reader.
\end{proof}

\begin{remark}
By arguments of the same kind as above, one can show that the
diagonal value for $x=y$ (part (1)) admits also another, kind of
``dual'', expression:
\begin{equation}
\Gamma(x,a,x,b,z)=
(z \land (x \lor b)) \lor (a \land x).
\label{modular'}
\end{equation}
The equality of these two expressions is equivalent to the
{\em modular law}
\begin{equation}
\Gamma(x,a,x,x,z)=
(z \land x) \lor (a \land x) = ((z \land x) \lor a) \land x.
\label{modular}
\end{equation}

It is known \cite{PR09} that any (finitely based) variety of lattices
can be axiomatized by a single \emph{quaternary} operation $q(\cdot,\cdot,\cdot,\cdot)$
given in terms of
the lattice operations by $q(x,b,z,a) = (z \land (x \lor b)) \lor (a \land x)$. 
That is, one may start with a quarternary operation $q$ satisfying certain identities 
(which we omit), define $x\lor y = q(y,x,x,y)$ and $x\land y = q(y,y,x,x)$, and the
resulting structure will be a lattice. From the preceding paragraph, we see that in 
our setting, $q(x,b,z,a) = \Gamma(x,a,x,b,z)$. Thus the quaternary approach to lattices
emerges from the present theory in a completely natural way.
\end{remark}

\begin{corollary}\label{lattice'}
\begin{enumerate}[label=\emph{(}\arabic*\emph{)},leftmargin=*]
\item
If $b \lor x = W$ and $a \land x = 0$, then
$\Gamma(x,a,x,b,z) 
= z$.
\item
If $a \lor y =W$ and $b \lor x = W$, then
$\Gamma(x,a,y,b,a)=a$.
\item
If $x \land a=0$ and $y\land b = 0$, then
$\Gamma(x,a,y,b,b)=b$.
\end{enumerate}
\end{corollary}

\begin{proof}
Straightforward consequences of the theorem,
again using the absorption laws.
\end{proof}

\subsection{Structural transformations and self-distributivity}

\emph{Homomorphisms} between sets with quintary product maps $\Gamma$,
$\Gamma'$ are defined in the usual way, and may serve to define the
category of Grassmannian geometries with their product maps $\Gamma$.
We call this the ``usual'' category.
There is another and often more useful way to turn them
into a category which we call ``structural'':

\begin{definition}
Let $W,W'$ be two right $\bB$-modules and $(\cX,\Gamma)$,
$(\cX',\Gamma')$ their Grassmannian geometries.
A \emph{structural} or \emph{adjoint pair of transformations between
$\cX$ and $\cX'$} is a pair of maps $f:\cX \to \cX'$,
$g:\cX' \to \cX$ such that, for all
$x,a,y,b,z  \in \cX$, $x',a',y',b',z' \in \cX'$,
\begin{align*}
f \big( \Gamma(x,g(a'),y,g(b'),z) \big) &= \Gamma'(f(x),a',f(y),b',f(z)),
\\
g \big( \Gamma'(x',f(a),y',f(b),z') \big)&= \Gamma(g(x'),a,g(y'),b,g(z'))\,.
\end{align*}
In other words, for fixed $a,b$, resp.\ $a',b'$, the restrictions
\[
f:\cX_{g(a'),g(b')} \to \cX_{a',b'}'\,, \quad
g:\cX_{f(a),f(b)}' \to \cX_{a,b}
\]
are homomorphisms of semitorsors. We will sometimes write $(f,f^t)$ for
a structural pair (although $g$ need not be uniquely determined by $f$).
\end{definition}

It is easily checked that the composition of  structural pairs gives
again a  structural pair, and  Grassmannian geometries with
structural pairs as morphisms form a category. Isomorphisms, and,
in  particular, the automorphism group of $(\cX,\Gamma)$, are
essentially
 the same in the usual and in the structural
 categories, but the endomorphism semigroups may be
very different. Roughly speaking, Grassmannian geometries tend to be
``simple objects'' in the usual category (hence morphisms tend to be
either trivial or injective), whereas they are far from being simple
in the structural category, so there are many morphisms.
One way of constructing such morphisms is via
ordinary $\bB$-linear maps $f:W \to W'$, which
induce maps between the corresponding
Grassmannians $\cX=\Gras(W)$ and $\cX'=\Gras(W')$:
\[
f_*:\cX \to \cX'; \, x \mapsto f(x), \qquad
f^*:\cX' \to \cX; \, y \mapsto f\inv(y).
\]
Note that, in general, these maps do not restrict to maps between
connected components (for instance,  $f_*$ and
$f^*$ do not restrict to \emph{everywhere} defined maps between projective
spaces $\bP W$ and $\bP W'$ if $f$ is not injective).
%
%
We will show that $(f_*,f^*)$ is an adjoint pair, as
a special case of the following result:

\begin{theorem}
Given a linear relation $\rr \subset W \oplus W'$, let
\begin{align*}
\rr_* : \cX \to \cX' ; &\quad x \mapsto \rr(x):=
\setof{ \omega' \in W'}{\exists \xi \in x : (\xi,\omega') \in \rr}\,, \\
\rr^* : \cX' \to \cX ; &\quad y \mapsto \rr^{-1}(y):=
\setof{ \omega \in W}{\exists \eta \in y : (\omega,\eta) \in \rr}\,.
\end{align*}
Then $(\rr_*,\rr^*)$ is a structural pair of transformations
between $\cX$ and $\cX'$.
\end{theorem}

\begin{proof}
Using the $(a,b)$-description, on the one hand,

\smallskip
\noindent $\rr_* \Gamma(x,\rr^* a',y,\rr^*b',z) =$
\begin{align*}
&= \Bigsetof{ \omega' \in W'}{\exists \omega \in \Gamma(x,\rr^* a',y,\rr^*b',z) :
(\omega,\omega') \in \rr} \\
&=
\Bigsetof{\omega' \in W'}{
\begin{array}{c}
\exists \omega \in W,
\exists \alpha \in \rr^* a', \exists \beta \in \rr^* b' : \\
(\omega,\omega') \in \rr,
\omega + \alpha \in z,
\omega + \beta \in x,
\omega + \alpha + \beta \in y
\end{array}} \\
&=
\Bigsetof{ \omega' \in W'}{
\begin{array}{c}
\exists \omega \in W,
\exists \alpha' \in  a', \exists \alpha \in W, \exists \beta' \in  b',
\exists \beta \in W : \\
(\omega,\omega') \in \rr, (\alpha,\alpha') \in \rr, (\beta,\beta') \in \rr,\\
\omega + \alpha \in z,
\omega + \beta \in x,
\omega + \alpha + \beta \in y
\end{array}}\,,
\end{align*}
and on the other hand,

\smallskip
\noindent $\Gamma(\rr_*x,a',\rr_*y,b',\rr_*z) =$
\begin{align*}
&= \Bigsetof{ \omega' \in W'}{
\begin{array}{c}
\exists  \alpha'' \in  a', \exists \beta'' \in b' :\\
\omega' + \alpha'' \in \rr^*z,
\omega' + \beta'' \in \rr^*x,
\omega' + \alpha'' + \beta'' \in \rr^*y
\end{array}} \\
&= \Bigsetof{ \omega' \in W'}{
\begin{array}{c}
\exists  \alpha'' \in  a', \exists \beta'' \in b', \exists
\zeta \in z, \exists \xi \in x, \exists \eta \in y : \\
(\zeta,\omega' + \alpha'') \in \rr,
(\xi,\omega'+ \beta'') \in \rr,
(\eta,\omega' + \alpha'' + \beta '') \in \rr
\end{array} }\,.
\end{align*}
The subspaces of $W$ determined by these two conditions are the
same, as is seen by the change of variables
\[
\zeta = \omega + \alpha, \, \,
\xi = \omega + \beta, \, \,
\eta = \omega + \alpha + \beta, \, \,
\alpha''=\alpha', \, \,
\beta''=\beta'
\]
in one direction, and
\[
\omega = \eta - \zeta - \xi , \alpha''=\alpha', \, \,
\beta''=\beta', \, \,
\alpha = \zeta - \omega  = \eta - \xi , \, \,
\beta = \xi - \omega  = \eta - \zeta
\]
in the other, and using that $\rr$ is a linear subspace.
\end{proof}

\begin{remark}
The proof shows that the same result would hold if we had
formulated the structurality property with  respect to
another ``admissible'' pair of variables instead of $(a,b)$,
for instance $(y,b)$ or $(x,z)$, by using the corresponding description.
However, we prefer to distinguish the pair formed by
the second and fourth variable
in order to have the interpretation of structural transformations in
terms of torsor homomorphisms, for fixed $(a,b)$.
\end{remark}

\begin{remark}
The construction from the theorem is functorial. In particular, the semigroup
of linear relations on $W \times W$ (to be more precise: a quotient with
respect to scalars) acts by structural pairs on $\cX$.
\end{remark}

\begin{theorem} \label{structure}
We define \emph{operators of left-, middle- and right multiplication
on $\cX$} by
$$
L_{xayb}(z):=R_{aybz}(x):=M_{xabz}(y):=\Gamma(x,a,y,b,z).
$$
Then, for all $x,a,y,b,z \in \cX$, the pairs
\[
(L_{xayb},L_{yaxb}), \quad (M_{xabz},M_{zabx}), \quad
(R_{aybz},R_{azby})
\]
are structural transformations of the Grassmannian geometry $\cX$.
\end{theorem}

\begin{proof}
Let $\bl_{x,a,y,b} \subset W \oplus W$ be the linear relation defined by
\[
\bl_{x,a,y,b}:=\setof{ (\zeta,\omega)\in W \oplus W}
{\exists \xi \in x : \omega + \zeta \in a, \omega + \zeta +\xi \in y,
\omega + \xi \in b}.
\]
Then it follows immediately by using the $(x,z)$-description that
\[
(\bl_{x,a,y,b})_*(z)= \setof{ \omega \in W}{\exists \zeta \in z :
(\zeta,\omega) \in \bl_{x,a,y,b} } = \Gamma(x,a,y,b,z) = L_{xayb}(z).
\]
On the other hand,
\begin{align*}
(\bl_{x,a,y,b})^*(z) &= \setof{  \omega \in W}{\exists \zeta \in z:
(\omega,\zeta) \in \bl_{x,a,y,b}} \\
&= \Bigsetof{  \omega \in W}{\exists \zeta \in z,\exists \xi \in x :
 \omega + \zeta \in a, \omega + \zeta +\xi \in y,
\zeta +\xi \in b } \\
& = \Gamma(y,a,x,b,z) = L_{yaxb}(z)\,,
\end{align*}
where the third equality follows by using the $(y,z)$-description with
permuted variables.
This proves that $(L_{xayb},L_{yaxb})$ is a structural pair; the claim
for right multiplications is just an equivalent version of this, and
the claim for middle multiplications is proved in the same way as above.
\end{proof}

\begin{remark}
We have proved that, in terms of inverses of linear  relations,
\begin{equation}
(\bl_{x,a,y,b})^{-1}= \bl_{y,a,x,b}.
\end{equation}
If $x \top a$ and $y \top b$, then $\bl_{xaby}$ is the graph of the linear
operator $L_{xayb} \in \End(W)$; for $x,y \in U_{ab}$, this operator is
invertible and
the preceding formula holds in the sense of an operator equation.
\end{remark}

\begin{corollary}
The multiplication map satisfies the following ``self-distributivity''
identities:
\begin{multline*}
\Gamma\Bigl(x,a,
\Gamma\bigl(u, \Gamma(a,z,c,x,b),v,\Gamma(a,z,d,x,b),w \bigr),
b,z \Bigr) = \\
\Gamma\Bigl(
\Gamma(x,a,u,b,z),c,\Gamma(x,a,v,b,z),d,\Gamma(x,a,w,b,z) \Bigr)
\end{multline*}
\begin{multline*}
\Gamma\Bigl(x,a,y,b,
\Gamma\big(u, \Gamma(y,a,x,b,c),v,\Gamma(y,a,x,b,d),w \bigr)
\Bigr) = \\
\Gamma\Bigl(
\Gamma(x,a,y,b,u),c,\Gamma(x,a,y,b,v),d,\Gamma(x,a,y,b,w) \Bigr)
\end{multline*}
\end{corollary}

\begin{proof}
The first identity follows by applying the adjoint pair
$(f,f^t)=(M_{xabz},M_{zabx})$
to $\Gamma(u,c,v,d,w)$ (and using the symmetry property), and
similarly the second by using the pair
$(f,f^t)=(L_{xayb},L_{yaxb}$).
\end{proof}

\begin{corollary} \label{pair!!} For all $a,b \in \cX$,
the  maps $(\quad )^+:U_a \times U_b \times U_a \to U_a$ and
$(\quad)^-:U_b \times U_a \times U_b \to U_b$ defined in Theorem \ref{pair!}
are tri-affine (i.e., affine in all three variables) and
satisfy the para-associative law
\[
(xy(uvw)^\pm)^\pm = ((xyu)^\pm vw)^\pm = (x(vuy)^\mp w)^\pm.
\]
\end{corollary}

\begin{proof}
Let us show that $M_{xabz}$ induces an affine map
$U_b \to U_a$, $y \mapsto (xyz)^+$,
for fixed $x,z \in U_a$. We know already that this map is well-defined
(Theorem \ref{pair!}).
Since
$(f,g)=(M_{xabz},M_{zabx})$ is structural, the map $f:U_{g(a)} \to U_a$
is  affine, where (according to Corollary \ref{lattice'}, (1)),
\[
g(a)=M_{zabx}(a)=\Gamma(z,a,a,b,x)=\Gamma(a,z,a,x,b)=b.
\]
By the same kind of argument, using Corollary \ref{lattice'}, (2) and (3),
wee see that the other partial maps are affine.
The corresponding statements for $(\quad)^-$ follow by symmetry,
and the para-associative law follows by restriction of the para-associative
law in the semitorsor $\cX_{ab}$.
\end{proof}

\begin{remark}
For $a=b$, we get the additive torsor $U_a$, and if $U_{ab} \not= \emptyset$,
then we get a sort of ``triaffine extension'' of the torsor $U_{ab}$.
If $a \top b$, then we have base points $a$ in $U_b$ and $b$ in $U_a$,
and obtain a trilinear product (Theorem \ref{pair!}).
\end{remark}

\subsection{The extended dilation map}
\sublabel{dilation}

Next we (re-)define, for $r \in \bK$, the \emph{dilation map}
$\Pi_r:\cX \times \cX \times \cX \to \cX$
by the following equivalent expressions
\begin{align*}
\Pi_r(x,a,z) &:=
\Bigsetof{\omega \in W}
{\exists \alpha \in a, \exists \zeta \in z, \exists \xi \in x: \,
\omega - r \alpha = \xi = \zeta - \alpha } \\
& =
\Bigsetof{\omega \in W}
{\exists \alpha \in a, \exists \zeta \in z, \exists \xi \in x: \,
\omega +(1-r)  \alpha = \zeta = \alpha + \xi } \\
& = \Bigsetof{\omega \in W}
{\exists \alpha \in a, \exists \zeta \in z, \exists \xi \in x: \,
\omega = (1-r)\xi + r \zeta, \, \zeta - \xi = \alpha}
\end{align*}
We refer to the last expression as the ``$(x,z)$-description'', and
we define partial maps $\cX \to \cX$ by
\[
\lambda^r_{xa}(z):= \rho^r_{az}(x):=\mu^r_{xz}(a):=\Pi_r(x,a,z)
\]
(where $\lambda$ reminds us of ``left'', $\rho$ ``right'' and $\mu$ ``middle'').

\begin{theorem}
The map $\Pi_r: \cX^3 \to \cX$ extends the ternary map
defined in the preceding chapter (and denoted by the same symbol
there), and it has the following properties:
\begin{enumerate}[label=\emph{(}\arabic*\emph{)},leftmargin=*]
\item \emph{Symmetry:} $\mu^r_{xz} = \mu^{1-r}_{zx}$, that is,
$\lambda^r_{xa}=\rho^{1-r}_{ax}$ or
\[
\Pi_r(x,a,z)=\Pi_{1-r}(z,a,x) .
\]
\item \emph{Multiplicativity:} if $x \top a$ and $r,s \in \bK$,
\[
\Pi_r(x,a,\Pi_s(x,a,y))= \Pi_{rs}(x,a,y),
\]
\item \emph{Diagonal values:}
\[
\Pi_r(x,a,x)=x, \quad
\Pi_0(x,a,z)= \Pi_1(z,a,x)=x \land (z \lor a)=\Gamma(a,x,x,a,z) .
\]
\item \emph{Structurality:} if $r(1-r) \in \bK^\times$,
then, for all $x,a,z \in \cX$, the pairs
\[
(\lambda_{xa}^r,\lambda_{ax}^r),\quad
(\mu_{xz}^r,\mu_{zx}^r)
\]
are structural transformations of $(\cX,\Gamma)$.
\end{enumerate}
\end{theorem}

\begin{proof}
The symmetry relation (1) follows directly from the
$(x,z)$-description.

Next we show that $\Pi_r$ coincides with the dilation map from the
preceding chapter.
Assume first that $x \top a$. We show that $\Pi_r(x,a,z)=(rP_a^x + P^a_x)(z)$:
\begin{align*}
(rP_a^x + P^a_x)(z) &=
\Bigsetof{ e \in W}{\exists \zeta \in z :
e = r P_a^x(\zeta) + P^a_x(\zeta)} \\
& = \Bigsetof{e \in W}{\exists \alpha \in a, \exists \zeta \in z, \exists \xi \in x:
e - r \alpha =  \zeta - \alpha = \xi }  = \Pi_r(x,a,z)
\end{align*}
writing $\zeta = P_a^x(\zeta) + P_x^a(\zeta)=\alpha + \xi$.
For $z \top a$, the claim follows now from the symmetry relation (1).

(3) With $\omega = (1-r)\xi + r \zeta$, it follows for $x=z$ that
$ \Pi_r(x,a,x) \subset  x$.
Conversely, we get $x \subset  \Pi_r(x,a,x)$ by letting  $\alpha =0$ and $\zeta=\xi$,
given $\xi \in x$.
 The other  relations are proved similarly.

(2) Under the assumption $x \top a$, the claim amounts to the operator
identity
\[
(rP_a^x + P^a_x)(sP_a^x + P^a_x) =
(rs P_a^x + P^a_x)
\]
which is easily checked.

(4) Fix $x,a \in \cX$, $r \in \bK$ and
define the linear subspace $\rr \subset W \oplus W$ by
\[
\rr:=\rr_{xa}:=
\big\{ (\zeta,\omega) \in W \oplus W | \,
\exists \alpha \in a, \exists \xi \in x : \,
\omega = \zeta - (1-r)\alpha, \,  \zeta - \alpha  =x \big\}
\]
Then
\[
\rr_*(z)=\{ \omega \in W | \, \exists \zeta \in z:
(\zeta,\omega) \in \rr \} = \Pi_r(x,a,z).
\]
On the other hand, by a straightforward change of variables (which is bijective
since $r$ is assumed to be invertible), one checks that
\[
\rr^*(z) = \setof{ \omega \in W}{\exists \zeta \in z :  (\omega,\zeta)
\in \rr} =  \Pi_r(a,x,z).
\]
Hence $(\lambda_{xa}^r,\lambda_{ax}^r) = (\rr_*,\rr^*)$ is structural.
%
%
%
The calculation  for the middle multiplications is similar.
%
\end{proof}

\noindent{\it Remarks.}
1.  If $r$ is invertible, then $\Pi_r(a,x,z)=\Pi_{r\inv}(x,a,z)$. Combining with (1), we see that $\Pi$
has the same behaviour under permutations as for the classical cross-ratio. 

2. 
If $x \top a$ and $r \in \bK$ an arbitrary scalar,
we still have structurality in (4). The situation is less clear
if $x,a,r$ are all arbitrary.

3.
One can define \emph{structurality with respect to $\Pi_r$} in the same
way as for $\Gamma$, by conditions of the form
\[
f\bigl( \Pi_r(x,g(a'),z \bigr)= \Pi_r'(f(x),a',f(z)), \quad
g\bigl( \Pi_r'(x',f(a),z' \bigr)= \Pi_r(g(x'),a,g(z')).
\]
Then partial maps of $\Gamma$ are structural for $\Pi_r$, and
partial maps of $\Pi_s$ are structural for $\Pi_r$ (this property
has been used in \cite{Be02} to characterize \emph{generalized
projective geometries}). The proofs are similar to the ones given above.

\section{Associative geometries}

In this chapter we give an axiomatic definition of associative
geometry, and we show that, at a base point, the corresponding
``tangent object'' is an associative pair. Conversely, given an
associative pair, one can reconstruct an associative geometry.
The question whether these constructions can be refined to give a
suitable equivalence of categories will be left for future work.

\subsection{Axiomatics}

\begin{definition}
An \emph{associative geometry} over a commutative unital ring $\bK$ is given
by a set $\cX$ which carries the following structures: $\cX$ is a
\emph{complete lattice} (with join denoted by $x \lor y$ and meet denoted by
$x \land y$), and
maps (where $s \in \bK$)
\[
\Gamma :\cX^5 \to \cX, \quad \Pi_s:\cX^3 \to \cX,
\]
such that the following holds.
We use the notation
\[
L_{xaby}(z):=M_{xabz}(y):=R_{aybz}(x):=\Gamma(x,a,y,b,z)
\]
for the partial maps of $\Gamma$, and call $x$ and $y$ \emph{transversal},
denoted by $x \top y$, if $x \land y = 0$ and $x \lor y = 1$, and we let
\[
C_a := a^\top := \setof{ x \in \cX}{x \top a}, \quad
C_{ab}:=C_a \cap C_b
\]
for sets of elements transversal to $a$, resp.\ to $a$ and $b$.
\begin{enumerate}[label=\emph{(}\arabic*\emph{)},leftmargin=*]
\item
\emph{The semitorsor property:}
 for all $x,y,z,u,v,a,b \in \cX$:
\[
\Gamma(\Gamma(x,a,y,b,z),a,u,b,v)=
\Gamma(x,a,\Gamma(u,a,z,b,y),b,v)=\Gamma(x,a,y,b,\Gamma(z,a,u,b,v)) .
\]
In other words, for fixed $a,b$,
 the product $(xyz):=\Gamma(x,a,y,b,z)$
turns $\cX$ into a semitorsor, which will be denoted by $\cX_{ab}$.
\item
\emph{Invariance of $\Gamma$ under the Klein $4$-group in $(x,a,b,z)$:}  for
all $x,a,y,b,z \in \cX$,
\begin{enumerate}[label=\emph{(}\roman*\emph{)},leftmargin=*]
\item $\Gamma(x,a,y,b,z)=\Gamma(z,b,y,a,x) $
\item $\Gamma(x,a,y,b,z)=\Gamma(a,x,y,z,b) $
\end{enumerate}
In particular,
 $\cX_{ba}$ is the opposite semitorsor of $\cX_{ab}$.
\item
\emph{Structurality of partial maps:}
for all $x,a,y,b,z \in \cX$, the pairs
\[
(L_{xayb},L_{yaxb}), \quad (M_{xabz},M_{zabx}), \quad
(R_{aybz},R_{azby})
\]
are structural transformations (see definition below).
\item
\emph{Diagonal values:}
\begin{enumerate}[label=\emph{(}\roman*\emph{)},leftmargin=*]
\item
for all $a,b,y \in \cX$,
$\Gamma(a,a,y,b,b)=a \lor b$,
\item
for all $a,b,y \in \cX$,
$\Gamma(a,b,y,a,b)=a \land b$,
\item
if $x \in C_{ab}$, then $\Gamma(x,a,x,b,z)=z=\Gamma(z,b,x,a,x)$,
\item
if $a \top x$ and $y \top b$, then
$\Gamma(x,a,y,b,b)= b$,
\item
%
if $a \top y$ and $b \top x$, then $\Gamma(x,a,y,b,a)= a$.
\end{enumerate}
\item
\emph{The affine space property:}
for all $a \in \cX$ and $r \in \bK$, $C_a$ is stable under the dilation map
$\Pi_r$, and $C_a$ becomes an affine space with additive torsor structure
$$
x - y +x = \Gamma(x,a,y,a,z)
$$
and scalar action given
for $x,y \in C_a$ by
\[
r \cdot_x y = (1-r) x + r y = \Pi_r(x,a,y).
\]
\item
\emph{The semitorsored pairs:} for all $a,b \in \cX$,
\[
\Gamma(U_a,a,U_b,b,U_a) \subset U_a, \quad
\Gamma(U_b,a,U_a,b,U_b) \subset U_b .
\]
\end{enumerate}
\end{definition}

\begin{definition}
The \emph{opposite geometry} of an associative geometry $(\cX,\top,\Gamma,\Pi)$,
denoted by $\cX^{op}$,
is $\cX$ with the same
dilation map $\Pi$, the opposite quintary product map
\[
\Gamma^{op}(x,a,y,b,z):=\Gamma(z,a,y,b,x)\,,
\]
(which by (4) induces the dual lattice structure) and transversality relation $\top$ 
determined by the lattice structure.
A \emph{base point} in $\cX$ is a fixed transversal pair $(o^+,o^-)$,
and the \emph{dual base point} in $\cX$ is then $(o^-,o^+)$.
\end{definition}

\begin{definition}
\emph{Homomorphisms of associative geometries} are maps
$\phi:\cX \to \cY$ such that
\begin{align*}
\phi \big( \Gamma(x,a,y,b,z) \big) &=
\Gamma(\phi x,\phi a,\phi y,\phi b,\phi z) \\
\phi \bigl( \Pi_r(x,a,y)) &= \Pi_r(\phi x,\phi a,\phi y) \bigl)
\end{align*}
It is clear that associative geometries over $\bK$ with their
homomorphisms form a category.
\emph{Antihomomorphisms} are homomorphisms into the opposite geometry.
Note that, by (4), homomorphisms are in particular lattice homomorphisms, and
antihomomorphisms are in particular lattice antihomomorphisms. 
\emph{Involutions} are antiautomorphims of order two; they play
an important r\^ole which will be discussed in subsequent work
\emph{(}\cite{BeKi09}\emph{)}.
For a fixed base point $(o^+,o^-)$, we
define the \emph{structure group} as the group of automorphisms of
$\cX$ that preserve $(o^+,o^-)$.

\medskip
An \emph{adjoint} or \emph{structural pair of transformations} is a pair
$g:\cX \to \cY$, $h:\cY \to \cX$
 such that
\begin{align*}
 g \bigl( \Gamma(x,h(u),y,h(v),z) \bigl)  &=  \Gamma( g(x),u,g(y),v,g(z) )  \\
g \bigl( \Pi_r (x,h(u),y \bigr) &= \Pi_r (g(x) ,u,g(y) )
\end{align*}
and vice versa.  Clearly, this also
defines a category.
\end{definition}

\subsection{Consequences}
\sublabel{consequences}

We are going to derive some easy consequences of the axioms.
%
%
%
Let us rewrite the semitorsor property in operator form:
\[
\begin{matrix}
R_{aubv} L_{xayb}&  =& M_{xabv}M_{uaby}& =& L_{xayb}R_{aubv}
\cr
L_{xayb} L_{zaub}& =& L_{x,a,L_{ybza}(u),b} & = &
L_{L_{xayb}(z),a,y,b}
\cr
M_{\Gamma(x,a,y,b,z),a,b,v} & =& M_{xabv} L_{ybza} & = & L_{xayb} M_{zabv}
\end{matrix}
\]
Assume that
$x,y \in U_{ab}$ and $z \in \cX$. Then, according to (4),
$L_{xaxb}=\id_\cX = L_{ybyb}$, whence
$
L_{xabb} L_{yaxb} = L_{x,a, L_{ybyb}(x),b} = L_{xaxb} = \id_\cX
$,
and $L_{xayb}:\cX \to \cX$ is invertible with inverse
\[
(L_{xayb})\inv = L_{yaxb}.
\]
By (2), this is equivalent to $(R_{aybx})\inv = R_{axby}$, and
in the same way one shows  that $M_{xaby}$ is invertible with inverse
\[
(M_{xaby})\inv = M_{xbay}.
\]
It follows that $L_{xayb}$, $R_{aybx}$ and $M_{xaby}$ are automorphisms
of the geometry.
In particular, $M_{xaay}$ and $M_{xabx}$ are of automorphisms of order two.

\begin{proposition}
For all $a,b \in \cX$, $C_{ab}$ is stable under the ternary map
$(x,y,z) \mapsto\Gamma(x,a,y,b,z)$, which turns it into a torsor
denoted by $U_{ab}$.
For any $y \in U_{ab}$, the group $(U_{ab},y)$ acts on $\cX$
 from the left and from the right by the formulas given
in Theorem 1.3, and  both actions commute.
\end{proposition}

\begin{proof}
As remarked above, $L_{xayb}$ is an automorphism of the geometry.
It stabilizes $a$ and $b$ and hence also $C_a$ and $C_b$.
Thus $C_{ab}$ is stable under the ternary map, and the para-associative
law and the idempotent law hold by (1) and (4) (iii).
The remaining statements follow easily from (1).
\end{proof}

In part II (\cite{BeKi09}) we will also describe the ``Lie algebra''
of $U_{ab}$, thus giving a relatively simple description of the
group structure of $U_{ab}$.
%
%
-- Next we give the promised conceptual interpretation of Equation (\ref{coinc}).

\begin{lemma} \label{coinc'}
For all $z \in U_b$, $x \in U_{ab}$,  and all $y \in \cX$,
\[
\Gamma\bigl(x,b,\Gamma(x,a,y,b,z),b,z\bigr)  =
\Gamma(z,b,a,y,x) .
\]
\end{lemma}

\begin{proof}
Using that $R_{xaxb} = \id_\cX$ for $a,b \in U_x$,
we have, for all $x,z \in U_b$,
\begin{eqnarray*}
\Gamma\bigl(x,b,\Gamma(x,a,y,b,z),b,z\bigr) & = &
M_{xbbz} M_{xabz}(y) \cr
&=& M_{bzxb}M_{bzxa}(y) \cr
&=& L_{bzax} R_{zbxb}(y) \cr
&=& L_{bzax} (y)
=\Gamma(b,z,a,x,y)=\Gamma(z,b,a,y,x).
\end{eqnarray*}
\end{proof}

\noindent
Since the operator $M_{xbbz}$ is invertible with inverse $M_{zbbx}$,
we have, equivalently,
\[
\Gamma(x,a,y,b,z)  =
\Gamma(z,b,\Gamma(z,b,a,y,x),b,x) .
\]
If $a$ and $b$ are transversal, we may
rewrite the lemma in the form (\ref{coinc}): with $b=o^-$, $y=o^+$: for all $x,z \in V^+$,
\[
\Gamma(z,o^-,a,o^+,x) =
\Gamma(x,o^-,\Gamma(x,a,o^+,o^-,z),o^-,z)=
x - \Gamma(x,a,o^+,o^-,z) + z.
\]
We will see in the following result that $\Gamma(z,o^-,a,o^+,x)$ is trilinear in $(z,a,x)$,
and hence $\Gamma(x,a,o^+,o^-,z)$ is tri-affine in $(x,a,z)$, and both expressions can be
considered as geometric interpretations of the associative pair attached to $(o^+,o^-)$
(see \S{0.4}).
More generally, the lemma implies
the following analog of Axiom (6): {\em for all $b,y \in \cX$ (transversal or not), the map
\[
U_b \times U_y \times U_b \to U_b,  \quad
(x,a,z) \mapsto \Gamma(x,a,y,b,z)
\]
is well-defined and affine in all three variables.}

\subsection{From geometries to associative pairs}
\sublabel{geom_to_pairs}

See Appendix B for the notion of \emph{associative pair}.

\begin{theorem}
Let $(\cX,\top,\Gamma,\Pi_r)$ be an associative geometry over
$\bK$.
\begin{enumerate}[label=\roman*\emph{)},leftmargin=*]
\item
Assume $\cX$ admits a transversal pair, which we take as
base point $(o^+,o^-)$.
Then, letting $\bA^+:=U_{o^-}$ and $\bA^-:=U_{o^+}$,
the pair of linear spaces
$(\bA^+,\bA^-)$ with origins $o^+$, resp.\ $o^-$, becomes an associative
pair when equipped with
\[
\langle xbz\rangle^+:= \Gamma(x,o^-,b,o^+,z), \quad
\langle ayc\rangle^-:= \Gamma(a,o^-,y,o^+,c).
\]
This construction is functorial (in the ``usual'' category).
\item
Assume $\cX$ admits a transversal triple $(a,b,c)$.
Then, letting $\bB:=U_c$, the $\bK$-module
$\bB$ with origin $o^+:=a$ becomes an associative unital
algebra with unit $u:=b$ and product map
\[
\bA \times \bA \to \bA, \quad (x,z) \mapsto
x z:= \Gamma(x,a,u,c,z).
\]
This construction is functorial (in the ``usual'' category) .
\end{enumerate}
\end{theorem}

\begin{proof}
(i)
By the ``semi-torsored pair axiom'' (6), the maps
$\bA^\pm  \times \bA^\mp \times \bA^\pm \to \bA^\pm $
are well-defined. By restriction from $\cX_{o^+,o^-}$, they
satisfy the para-associative law.
They are tri-affine: the proof is exactly the same as the one of
Corollary \ref{pair!!}.
Thus it only remains to be shown that they are trilinear,
with respect to the origins $o^\pm \in \bA^\pm$.
Let $x,z \in \bA^+$ and $b \in \bA^-$.
Then
\[
\langle xbo^+\rangle^+=\Gamma(x,o^-,b,o^+,o^+)=o^+, \quad
\langle o^+ bz\rangle^+= \Gamma(o^+,o^-,b,o^+,z)=o^+
\]
\[
\langle xo^-z\rangle^+=\Gamma(x,o^-,o^-,o^+,z)=\Gamma(o^-,x,o^-,z,o^+)=o^+.
\]
by the Diagonal Value Axiom (4).
If $\phi:\cX \to \cY$ is a base-point preserving homomorphism,
then restriction of $\phi$ yields, by definition of a homomorphism,
a pair of $\bK$-linear maps $\bA^\pm \to (\bA')^\pm$,
which commutes with the product maps $\Gamma,\Gamma'$ and hence
is a homomorphism of associative pairs.

(ii)
With notation from (i), we have $xz=\langle xuz\rangle^+$, and hence the product
is well-defined, bilinear and associative  $\bA \times \bA \to \bA$.
We only have to show that $u$ is a unit element:
but this is immediate from
$xu=\Gamma(x,a,u,b,u)=x=\Gamma(u,a,u,b,x)=ux$.
\end{proof}

\begin{example}
For any $\bB$-module $W$, the Grassmannian geometry $\cX$ is an
associative geometry, by the results of Chapter 2.
For a decomposition $W=o^+ \oplus o^-$, the corresponding
associative pair is
$(\bA^+,\bA^-)=
(\Hom_\bB(o^+,o^-),\Hom_\bB(o^-,o^+))$, by Theorem 1.7.
In case $W$ is a topological module over a topological ring $\bK$,
we may also work with subgeometries of the whole Grassmannian, such
as Grassmannians of {\em closed subspaces with closed complement}.
For $\bK=\bR$ or $\bC$, if $W$ is, e.g., a Banach space, the associated
associative pair is a {\em pair of spaces of bounded linear operators}.
\end{example}

\begin{remark}
There is a natural definition of \emph{structural transformations of
associative pairs}. They are induced by structural pairs $(f,g)$
satisfying $f(o^+)=o^+$, $g(o^-)=o^-$ and $f(\bA^+)\subset \bA^+$,
$g(\bA^-) \subset \bA^-$. With respect to such pairs, the construction
obtained from the theory is still functorial.
\end{remark}

\subsection{From associative pairs to geometries}

\begin{theorem}
\begin{enumerate}[label=\roman*\emph{)},leftmargin=*]
\item
For every associative pair $(\bA^+,\bA^-)$ there exists
an associative geometry $\cX$ with base point $(o^+,o^-)$ having
$(\bA^+,\bA^-)$ as associated pair.
\item
For every unital associative algebra $(\bA,1)$ there exists
an associative geometry $\cX$ with transversal triple $(o^+,\Delta,o^-)$ having
$(\bA,1)$ as associated algebra.
\end{enumerate}
\end{theorem}

\begin{proof}
(ii)
Let $W = \bA \oplus \bA$, $o^+$ the first and $o^-$
the second factor and $\Delta$ the diagonal.
Then $(o^+,\Delta,o^-)$ is a transversal triple in the Grassmannian
geometry $\cX = \Gras_\bA(W)$, and its associated
algebra is $\bA \cong \Hom_\bA(\bA,\bA)$ (see the preceding
example, with $o^+ \cong o^- \cong \bA$).
Note that the connected component of $o^+$ can be interpreted as
the \emph{projective line over $\bA$}, cf.\ \cite{BeNe05}, \cite{Be08}.

\smallskip
(i)
Consider any algebra imbedding $(\hat{\bA},e)$ of the pair
$(\bA^+,\bA^-)$, for instance, its standard imbedding (see Appendix B).
Let $e$ denote the idempotent giving the grading of $\hat{\bA}$
and set $f = 1-e$. 
Then
$\hat{\bA} = \bA_{00} \oplus \bA_{01} \oplus \bA_{10} \oplus \bA_{11}$
where $\bA_{00} = f\hat{\bA}f$, $\bA_{01} = \bA^- = f\hat{\bA}e$,
$\bA_{10} = \bA^+ = e\hat{\bA}f$ and $\bA_{11} = e\hat{\bA}e$.
Let $\cX = \Gras_{\hat{\bA}}(\hat{\bA})$ be the Grassmannian of
all \emph{right} ideals in $\hat{\bA}$. As base point in $\cX$ we
choose
\[
o^+ := e \hat{\bA} = \bA_{11} \oplus \bA_{10}, \quad
o^- := f \hat{\bA} = \bA_{00} \oplus \bA_{01}\,.
\]
The associative pair corresponding to $(\cX;o^+,o^-)$ is
(see example at the end of the last section)
\[
(\Hom_{\hat{\bA}}(o^-,o^+),\Hom_{\hat{\bA}}(o^+,o^-) ).
\]
But this pair is naturally isomorphic to $(\bA^+,\bA^-)$.
Indeed,
\[
\Hom_{\hat{\bA}}(e \hat{\bA},f \hat{\bA}) \to \bA_{01} =  f \hat{\bA} e,
\quad
f \mapsto f(e)
\]
is $\bK$-linear, well-defined (since $f(e)e=f(ee)=f(e)$, $f$ being
right $\hat{\bA}$-linear) and has as inverse mapping
$c \mapsto (x \mapsto cx)$, hence is a $\bK$-isomorphism.
Identifying both pairs of $\bK$-modules in this way, a direct check
shows that the triple products also coincide, thus establishing the
desired isomorphism of associative pairs.
\end{proof}

\begin{remark}
It is of course
also possible to see (ii) as a special case of (i).
In this case we may work with the algebra imbedding of
$(\bA,\bA)$ into the matrix algebra
$\hat{\bA} = M(2,2;\bA)$, \textit{cf.}\ Appendix B.
\end{remark}

\begin{remark}[\textbf{Functoriality}]
Is the construction from the preceding theorem functorial, or can it
be modified such that it becomes functorial?
In the present form, the construction depends on the chosen algebra imbedding and hence is
not functorial (even if we always chose the standard imbedding the construction would not
become functoriel, see \cite{Ca04}).  However, motivated by corresponding results
from Jordan theory (\cite{Be02}),
we conjecture that the {\em geometry generated by the connected component}
depends functorially on the associative pair, thus leading to an equivalence of categories between associative pairs
and certain  associative
geometries with base point (whose algebraic properties reflect
``connectedness and simple connectedness'').
\end{remark}

\section{Further topics}
\seclabel{further}

\subsubsection*{(1) Jordan geometries revisited}
The present work sheds new light on geometries associated to \emph{Jordan
algebraic structures}:
in the same way as associative pairs give rise to Jordan pairs
by restricting to the diagonal ($Q(x)y=\langle xyx\rangle$; see Appendix B),
associative geometries give rise to ``Jordan geometries''. The new feature
is that we get \emph{two}
diagonal
restrictions $\Gamma(x,a,y,b,x)$ and $\Gamma(x,a,y,a,z)$ which are
equivalent. They  can be used to give a new axiomatic foundation of
``Jordan geometries''. Unlike the theory developed in
\cite{Be02}, this new foundation will
be valid also in case of characteristic 2 and
hence corresponds to general quadratic Jordan pairs. In this theory,
the torsors from the associative theory
 will be replaced by \emph{symmetric spaces} (the diagonal
$(xyx)$).

\subsubsection*{(2) Involutions, Jordan-Lie algebras, classical groups}
From a Lie theoretic point of view, the present work deals with
classical groups of type $A_n$ (the ``general linear'' family).
The other classical series (orthogonal, unitary and symplectic
families) can be dealt with by adding an \emph{involution} to
an associative geometry. This will be discussed in detail in
\cite{BeKi09}.
From a more algebraic point of view, this amounts to looking
at \emph{Jordan-Lie} or \emph{Lie-Jordan algebras} instead of
associative pairs (and hence is closely related to (1)),
and asking for the geometric counterpart.
In \cite{Be08}, it is advocated that this might also be interesting
in relation with foundational issues of quantum mechanics.

\subsubsection*{(3) Tensor Products}
In the associative and in the
Jordan-Lie categories, \emph{tensor products} exist (cf.\ \cite{Be08}
for historical remarks on this subject in relation with foundations of
Quantum Mechanics). What is the geometric interpretation of this
remarkable fact?

\subsubsection*{(4) Alternative Geometries}
The geometric object corresponding to \emph{alternative pairs}
(see \cite{Lo75}) should be a collection of Moufang loops,
interacting among each other in a similar way as the torsors
$U_{ab}$ do in an associative geometry.

\subsubsection*{(5) Classical projective geometry revisted}
The torsors $U_{ab}$ show already up in ordinary projective spaces,
and their alternative analogs will show up in octonion projective
planes. It should be interesting to review classical approaches
from this point of view.

\subsubsection*{(6) Invariant Theory}
The problem of
classifying the torsors $U_{ab}$ in a given geometry $\cX$ is
very close to classifying orbits in $\cX \times \cX$ under the
automorphism group. Invariants of torsors (``rank'') give rise to invariants
of pairs. Similarly, invariants of groups $(U_{ab},y)$ give rise to
invariants of triples (``rank and signature''), and invariants
(conjugacy class) of projective endomorphisms $L_{xayb}$ to
invarinats of quadrupels (``cross-ratio'').

\subsubsection*{(7) Structure theory: ideals and intrinsic subspaces}
We ask to translate features of the structure theory of associative pairs and
algebras to the level of associative geometries: what are
the geometric notions corresponding to left-, right- and inner ideals?
See \cite{BeL08} for the Jordan case.

\subsubsection*{(8) Positivity and convexity: case of $C^*$-algebras}
$C^*$-algebras and related triple systems (``ternary rings of operators'',
see \cite{BM04}) are distinguished among general ones by properties
involving ``positivity'' and ``convexity''. What is their geometric
counterpart on the level of associative geometries?
Note that these properties really belong to the involution $*$, so
these questions can be seen to fall in the realm of topic (2).

\section*{Appendix A: torsors and semitorsors}

\begin{definition}
A \emph{torsor} $(G,(\cdot,\cdot,\cdot))$ is a set $G$ together
with a ternary operation
$G^3 \to G; (x,y,z) \mapsto (xyz)$ satisfying the identities (G1)
and (G2) discussed in the Introduction (\S{0.2}).
\end{definition}

An early term for this notion, due to
Pr\"{u}fer, was \emph{Schar}. This was translated by
Suschkewitsch into Russian as {\cyr{grud}}. This was later
somewhat unfortunately translated into English as ``heap''.
Other terms that have been used are ``flock'' and ``herd''.
B. Schein, in various publications (\emph{e.g.}, \cite{Sch62})
and in private communication,
suggested adapting the Russian term directly
into English as ``groud'' (this rhymes with ``rude'', not ``crowd'').
In earlier drafts of this paper, we followed his suggestion.
However, we found that in talks on this subject, the general audience reaction to the term
was negative, and the referee noted that it is hard to pronounce
in an English sentence. The other terms noted above are not
appropriate, either. In the follow-up \cite{BeKi09} to this
paper, we wish to write of ``classical'' objects just as we speak
of classical groups, and ``classical heap'' or ``classical herd'',
for instance, do not seem suitable.

In the end, we decided to go with \emph{torsor}. This term is usually
used in a geometric sense to mean \emph{principal homogeneous space},
and is now generally accepted, largely due to the
popularizing efforts of Baez \cite{Ba09}. Using the same term for the
equivalent algebraic notion seemed to us a quite reasonable step.

For more on the history of the concept,
as well as of what we call semitorsors defined below, we refer the reader
to the work of Schein, \emph{e.g.}, \cite{Sch62}.

In a torsor $(G,(\cdot,\cdot,\cdot)$,
introduce \emph{left-}, \emph{right-} and \emph{middle multiplications} by
\[
(xyz) =: \ell_{x,y}(z)=:r_{y,z}(x)=:m_{x,z}(y)\,.
\]
Then the axioms of a torsor can be rephrased as follows:
\begin{align*}
\ell_{x,y} \circ r_{u,v} = r_{u,v} \circ \ell_{x,y} \tag{G1'} \\
\ell_{x,x}= r_{x,x}=\id \tag{G2'}
\end{align*}
or, in yet another way,
\begin{align*}
\ell_{x,y} \circ \ell_{z,u} = \ell_{\ell_{x,y}(z),u} \tag{G1''} \\
\ell_{x,y}(y)= r_{y,x}(y)=x\,. \tag{G2''}
\end{align*}
Taking $y=z$ in (G1''), and using (G2''), we get what one might call ``Chasle's relation''
for left translations
\[
\ell_{x,y} \circ \ell_{y,u} = \ell_{x,u}
\]
which for $u=x$ shows that the inverse of $\ell_{x,y}$ is $\ell_{y,x}$.
Similarly, we have a Chasle's relation for right translations, and
the inverse of $r_{x,y}$ is $r_{y,x}$.
Unusual, compared to group theory, is the r\^{o}le of the middle multiplications.
Namely, fixing for the moment a unit $e$, we have
\[
(x(uyw)z)=x(uy\inv w)\inv z=xw\inv yu\inv z=((xwy)uz)=(xw(yuz))
\]
(the \emph{para-associative law}, cf.\ relation (G3), Introduction), i.e.,
\[
m_{x,z} \circ m_{u,w} = \ell_{x,w} \circ r_{u,z} = r_{u,z}\circ \ell_{x,w}.
\tag{G3'}
\]
Taking $x=w$, resp.\ $u=z$, we see that all left and right multiplications
can be expressed via middle multiplications:
\[
r_{u,z} = m_{x,z} \circ m_{u,x}, \quad \ell_{x,w}=m_{x,z}\circ m_{z,w} .
\]
Taking $u=z$, resp.\ $x=w$, we see that
 $m_{x,z} \circ m_{z,x} =\id$,
hence middle multiplications are invertible.
In particular $m_{x,x}^2=\id$, which reflects the fact
that $m_{x,x}$ is inversion in the group $(G,x)$.
Also, (G3') implies that
\[
m_{x,e} \circ m_{x,e} = r_{x,e} \circ \ell_{x,e} =
\ell_{x,e} \circ (r_{e,x})\inv,
\]
which means that
conjugation by $x$ in the group with unit $e$ is equal
to  $(m_{x,e})^2$.

Since a torsor can be viewed as an equational class in the sense of
universal algebra
$(G,(\cdot,\cdot,\cdot))$, all of the usual notions apply. For instance,
a \emph{homomorphism of torsors} is a map $\phi : G \to H$
such that $\phi\big((xyz)\big)=(\phi(x)\phi(y)\phi(z))$, and
an \emph{anti-homomorphism of torsors} is a homomorphism
to the \emph{opposite torsor} (same set with product
$(x,y,z) \mapsto (zyx)$).
Homomorphisms enjoy similar properties as usual affine maps.
It is easily proved that left and right multiplications are
automorphisms (called \emph{inner}), whereas middle multiplications
are inner anti-automorphisms.
Other notions, such as subtorsors, products, congruences and quotients
follow standard patterns.

\begin{definition}
A \emph{semitorsor} $(G,(\cdot,\cdot,\cdot))$ is a set $G$ with a ternary operation
$G^3 \to G; (x,y,z) \mapsto (xyz)$ satisfying
the para-associative law (G3) from the Introduction.
\end{definition}

The basic example is the \emph{symmetric semitorsor} on sets $A$
and $B$, the set of all relations between $A$ and $B$
with $(rst)=r\circ s\inv\circ t$, where $\circ$ is the
composition of relations.

Clearly, fixing the middle element in a semitorsor gives rise to
a semigroup; but, in contrast to the case of groups, not all
semigroups are obtained in this way. For more on semitorsors,
see, \emph{e.g.} \cite{Sch62} and the references therein.

\section*{Appendix B: Associative pairs}

\begin{definition}
An \emph{associative pair (over $\bK$)} is a pair $(\bA^+,\bA^-)$ of
$\bK$-modules together with two trilinear maps
\[
\langle \cdot,\cdot,\cdot \rangle^\pm :\bA^\pm \times \bA^\mp \times \bA^\mp \to
\bA^\pm
\]
such that
\[
\langle xy \langle zuv\rangle^\pm \rangle^\pm =
\langle\langle xyz\rangle^\pm uv\rangle^\pm=\langle x\langle uzy\rangle^\mp v\rangle^\pm.
\]
\end{definition}

\noindent Note that we follow here the convention of Loos \cite{Lo75}.
Other authors (e.g. \cite{MMG})
use a modified identity, replacing the last term
by $\langle x\langle yzu\rangle^\mp v\rangle^\pm$.
But both versions are equivalent: it suffices to replace
$\langle \quad \rangle^-$ by the trilinear map
$(x,y,z) \mapsto \langle z,y,x\rangle^-$. We prefer the definition given by
Loos since it takes the same form as the para-associative law
in a semitorsor.
We should mention, however, that for
\emph{associative triple systems}, i.e., $\bK$-modules
 $\bA$ with
a trilinear map $\bA^3 \to \bA$, $(x,y,z) \mapsto \langle xyz\rangle$
these two versions of the defining identity have to be distinguished,
leading to two different kinds of associative triple systems
(``ternary rings'', cf.\ \cite{Li71}, and associative triple systems
\cite{Lo72};
all this is best discussed in the context of associative pairs, resp.\
geometries, \emph{with involution}, see topic (2) in Chapter 4 and
\cite{BeKi09}.)
In any case, for fixed $a \in \bA^-$, $\bA^+$ with
\[
x \cdot_a y := \langle xay\rangle
\]
is an associative algebra, called the \emph{$a$-homotope} and
 denoted by $\bA_a^+$.

\subsubsection*{Examples of associative pairs}

\begin{enumerate}[label={(}\arabic*{)},leftmargin=*]
\item
Every associative algebra $\bA$ gives rise to an associative pair
$\bA^+ = \bA^- = \bA$ via $\langle xyz\rangle^+ = xyz$, $\langle xyz\rangle^- = zyx$.
\item
For  $\bK$-modules $E$ and $F$, let
$\bA^+ = \Hom(E,F)$, $\bA^- = \Hom(F,E)$,
\[
\langle XYZ\rangle^+ = X \circ Y \circ Z \qquad
\langle XYZ\rangle^- = Z \circ Y \circ X.
\]
\item
Let $\hat{\bA}$ be an associative algebra with unit $1$ and idempotent $e$
and $f:=1-e$.
Let
\[
\hat{\bA} = f \hat{\bA} f \oplus f \hat{\bA} e \oplus e \hat{\bA} e
\oplus e \hat{\bA} f =
\bA_{00} \oplus \bA_{01} \oplus \bA_{11} \oplus \bA_{10}
\]
with
$\bA_{ij} = \setof{x \in \hat{\bA}}{ex=ix, xe=jx}$
the associated Peirce decomposition. Then
\[
(\bA^+,\bA^-) := (\bA_{01},\bA_{10}), \quad
\langle xyz\rangle^+ := xyz, \quad \langle xyz\rangle^- := zyx
\]
is an associative pair.
\end{enumerate}

\subsubsection*{The standard imbedding}
It is not difficult to show that every associative pair arises from
an associative algebra $\hat{\bA}$ with idempotent $e$ in the way
just described (see  \cite{Lo75}, Notes to Chapter II).
We call this an \emph{algebra imbedding for $(\bA^+,\bA^-)$}.
There are several such imbeddings (see \cite{Ca04} for a comparison of
some of them). Among these is a minimal choice called the
\emph{standard imbedding of the associative pair}. 
For instance, in Example (2) we may
take $\hat{\bA} = \End(E \oplus F)$ with $e$ the projector onto $E$ along
$F$ (but this choice will in general not be minimal).
In Example (1), take $\hat{\bA} := \End_\bA(\bA \oplus \bA)=
M(2,2;\bA)$ and
$e$ the projector onto the first factor.

\subsubsection*{The associated Jordan pair}
Formally, associative pairs give rise to Jordan pairs in exactly the
same way as torsors give rise to symmetric spaces: the Jordan pair is
 $(V^+,V^-):=(\bA^+,\bA^-)$ with
the quadratic map
$Q^\pm(x)y:=\langle xyx\rangle^\pm$
and its polarized version
\[
T^\pm(x,y,z):=Q^\pm(x+z)y -Q^\pm(x)y- Q^\pm(z)y =
\langle xyz\rangle^\pm+\langle zyx\rangle^\pm.
\]

\subsubsection*{Associative pairs with invertible elements}
We call $x \in \bA^\pm$ \emph{invertible} if
\[
Q(x):\bA^\mp \to \bA^\pm, \quad y \mapsto \langle xyx\rangle
\]
is an invertible operator.
As shown in \cite{Lo75}, associative pairs with invertible
elements correspond to unital associative algebras:
namely, $x$ is invertible if and only if the algebra $\bA_x$ has
a unit (which is then $x\inv:=Q(x)\inv x$).


\begin{thebibliography}{BeNe05}

\bibitem[Ba09]{Ba09}
J.~Baez,
Torsors made easy,
\url{http://math.ucr.edu/home/baez/torsors.html}.

\bibitem[Be02]{Be02}
W.~Bertram,
Generalized projective geometries: general theory and equivalence with Jordan structures,
\textit{Adv. Geom.} \textbf{2} (2002), 329--369 (electronic version: preprint 90 at
\url{http://homepage.uibk.ac.at/~c70202/jordan/index.html}).

\bibitem[Be03]{Be03}
W.~Bertram,
The geometry of null systems, Jordan algebras and von Staudt's Theorem,
\textit{Ann. Inst. Fourier} \textbf{53} (2003) fasc. 1, 193--225.

\bibitem[Be04]{Be04}
W.~Bertram, From linear algebra via affine algebra to projective algebra,
\textit{Linear Algebra and its Applications} \textbf{378} (2004), 109--134.

\bibitem[Be08b]{Be08b}
W.~Bertram,
Homotopes and conformal deformations of symmetric spaces.
\textit{J. Lie Theory} \textbf{18} (2008), 301--333;
arXiv: \url{math.RA/0606449}.

\bibitem[Be08]{Be08}
W.~Bertram,
On the Hermitian projective line as a home for the geometry of quantum theory.
In \textit{Proceedings XXVII Workshop on Geometrical Methods in Physics,
Bia\l owie\.za 2008};
arXiv: \url{math-ph/0809.0561}.

\bibitem[Be10]{Be07}
W.~Bertram,
Jordan structures and non-associative geometry.
In \textit{Trends and Developments in Infinite Dimensional Lie Theory},
Progress in Math., Birkhaeuser, to appear 2010;
arXiv: \url{math.RA/0706.1406}.

\bibitem[BeKi09]{BeKi09}
W.~Bertram and M.~Kinyon,
Associative Geometries. II: Involutions, the classical torsors, and their homotopes,
\textit{J. Lie Theory}, to appear;
arXiv: \url{0909.4438}.

\bibitem[BeL08]{BeL08}
W.~Bertram and H.~Loewe,
Inner ideals and intrinsic subspaces,
\textit{Adv. in Geometry} \textbf{8} (2008), 53--85;
arXiv: \url{math.RA/0606448}.

\bibitem[BeNe05]{BeNe05}
W.~Bertram and K.-H.~Neeb,
Projective completions of Jordan pairs. II: Manifold structures and symmetric spaces,
\textit{Geom. Dedicata} \textbf{112} (2005), 73 -- 113;
arXiv: \url{math.GR/0401236}.

\bibitem[BM04]{BM04}
D.~P.~Blecher and C.~Le~Merdy,
\textit{Operator algebras and their modules -- an operator space approach},
Clarendon Press, Oxford, 2004.

\bibitem[Cer43]{Cer43}
J.~Certaine,
The ternary operation $(abc)=ab\sp {-1}c$ of a group,
\textit{Bull. Amer. Math. Soc.} \textbf{49} (1943), 869--877.

\bibitem[Ca04]{Ca04}
I.~de~las~Pe\~{n}as~Cabrera,
A note on the envelopes of an associative pair,
\textit{Comm. Algebra} \textbf{32} (2004), 4133--4140.

\bibitem[Li71]{Li71}
W.~G.~Lister,
Ternary rings,
\textit{Trans. Amer. Math. Soc.} \textbf{154} (1971), 37--55.

\bibitem[Lo69]{Lo69}
O.~Loos,
\textit{Symmetric Spaces I},
Benjamin, New York, 1969.

\bibitem[Lo72]{Lo72}
O.~Loos,
Assoziative Tripelsysteme,
\textit{Manuscripta Math.} \textbf{7} (1972), 103--112.

\bibitem[Lo75]{Lo75}
O.~Loos,
\textit{Jordan Pairs},
Lecture Notes in Math. \textbf{460}, Springer, New York, 1975.

\bibitem[MMG]{MMG}
J.~A.~Cuenca Mira, A.~Garc\`{i}a Mart\`{i}n and C.~Mart\`{i}n Gonz\'{a}lez,
Jacobson density for associative pairs and its applications,
\textit{Comm. Algebra} \textbf{17} (1989), 2595--2610.

\bibitem[PR09]{PR09}
R. Padmanabhan and S. Rudeanu,
\textit{Axioms for Lattices and Boolean Algebras},
World Scientific, 2009.

\bibitem[Sch62]{Sch62}
B.~Schein,
On the theory of inverse semigroups and generalized grouds,
in \textit{Twelve papers in logic and algebra}, AMS Translations, Ser. 2, \textbf{113},
1979, pp. 89--123.

\bibitem[Va66]{Va66}
V.~V.~Vagner,
On the algebraic theory of coordinate atlases.  (Russian)
\textit{Trudy Sem. Vektor. Tenzor. Anal.} \textbf{13} (1966), 510--563.

\end{thebibliography}
\end{document}